\documentclass[11pt,a4paper]{amsart}
\usepackage[a4paper,left=2.4cm,right=2.4cm,top=2.6cm,bottom=2.6cm]{geometry}
\usepackage{microtype}
\usepackage{amsmath,amssymb,amsthm}
\usepackage{stmaryrd}
\usepackage{graphicx}
\usepackage{adjustbox}
\usepackage{tikz}
\usetikzlibrary{arrows.meta,calc,patterns,positioning,shapes.geometric}
\usepackage{tikz-cd}
\usepackage[numbers,sort&compress]{natbib}
\usepackage[colorlinks=true,linkcolor=blue!50!black,citecolor=blue!50!black,urlcolor=blue!50!black]{hyperref}

\numberwithin{equation}{section}
\theoremstyle{plain}
\newtheorem{theorem}{Theorem}
\newtheorem{proposition}{Proposition}
\newtheorem{lemma}[theorem]{Lemma}
\theoremstyle{definition}
\newtheorem{definition}{Definition}
\theoremstyle{remark}
\newtheorem{rmk}{Remark}
\newenvironment{pf}{\begin{proof}}{\end{proof}}

\title{Schubert cells for the super affine Grassmannian of $GL_{n|m}$}
\author{S. Gureva}
\address{Moscow Institute of Physics and Technology, 9 Institutskiy per., Dolgoprudny, Moscow Region, 141701, Russia}
\email{gureva.sd@phystech.edu}
\thanks{Corresponding author.}
\author{V. Stukopin}
\address{Moscow Institute of Physics and Technology, 9 Institutskiy per., Dolgoprudny, Moscow Region, 141701, Russia}
\email{stukopin.va@mipt.ru}

\begin{document}
\begin{abstract}
We study Schubert cells in the affine Grassmannian of the general linear supergroup $GL_{n|m}$, focusing explicitly on the low-dimensional cases $n|m = 1|1$ and $2|1$. We define the affine super Grassmannian as an fpqc sheaf, describe its bounded-lattice filtration by finite-dimensional superschemes, and provide coordinate descriptions, equations, and morphisms for the cells that arise in low dimensions. For $GL_{2|1}$ the cells in the positive and negative chambers split as odd vector bundles over jet schemes of the projective line, and cyclic cells of $GL_{n|m}$ are jet schemes of projective superspaces. In the mixed-sign chamber the transition functions acquire nilpotent corrections in the even coordinate, and all cells except one are non-split. We also introduce decorations, which record the odd directions of the cells on the lattice graph.
\end{abstract}
\maketitle

\section{Introduction}\label{intro}
The affine Grassmannian $Gr_G$ of an algebraic group $G$ is an infinite-dimensional analog of a partial flag variety. It is related to representation theory through the geometric Satake correspondence, proved in \cite{mirkovic_2007_geometric}. Important components of this correspondence are the Schubert varieties \cite{zhu_2016_introduction}; see \cite{peerenboom2021affine} for a more explicit description.

Research on the Satake equivalence for supergroups has been conducted in \cite{Braverman_2021}. That work operates at the abstract level of categories and sheaves and discusses questions concerning Gaiotto conjectures and the Langlands correspondence. We pursue more down-to-earth goals, focusing on an explicit coordinate description of super Schubert cells for the supergroup $GL_{n|m}$.

This paper explores the super analog of the affine Grassmannian. More specifically, we construct and study the supervarieties that naturally emerge in this context. We begin by defining the affine super Grassmannian as an fpqc sheaf and explaining its ind-superscheme structure in Section \ref{prelim}. We then formulate and prove our results in Section \ref{results}. In Theorem \ref{gen}, we describe the general structure of a super Schubert cell in the affine super Grassmannian of $GL_{n|m}$. In Theorem \ref{th11}, we give a coordinate description for $GL_{1|1}$, while Theorems \ref{th21pos} and \ref{th21mix} give coordinate descriptions of the Schubert cells corresponding to coweights in the positive and mixed-sign chambers, respectively, where the coweights are parametrized by the classical Cartan decomposition on complex points, Proposition \ref{Cartan} (Remark \ref{rmk:tails} explains why the decomposition is stated on complex points only). Theorem \ref{th21nonsplit} shows that the mixed-sign cells are, with one exception, non-split. Section \ref{jets} identifies cyclic cells of $GL_{n|m}$ with jet schemes of projective superspaces (Theorem \ref{thcyclic}) and deduces a jet-scheme description of the cells of $GL_{1|1}$ and of $GL_{2|1}$ outside the mixed-sign chamber (Proposition \ref{prop:jet-cells}); decorations of the lattice graph are introduced in Definition \ref{def:decoration} and studied in Proposition \ref{prop:decoration}.

The explicit description of the Schubert cells in dimension $2|1$ reveals a structure that can be interpreted as the super analog of a jet scheme (see \cite{Mourtada2023}; jets and Weil bundles of supermanifolds as such have been studied, for example, in \cite{Bruce2014}). We are not aware of an earlier appearance of the identification of Schubert cells of the affine super Grassmannian with such jet schemes, although our search of the literature was limited. We use the term \emph{Schubert variety} only for the closure of a Schubert cell; the present paper concerns the cells themselves.
 
\section{Preliminaries}\label{prelim}
General references on supergeometry can be found in \cite{koblitz2013gauge} and \cite{carmeli2011mathematical}.

A supercommutative ring $R = R_0\oplus R_1$ is a superring (i.e. $\mathbb{Z}_2$-graded ring) such that for any two homogeneous elements $a$ and $b$ we have: 
\begin{align}\label{supcomm}
\forall a, b \in R \ \ \ \ \  ab = (-1)^{|a||b|}ba.
\end{align}
Elements that belong to $R_0$ (resp. $R_1$) are called even (resp. odd). We also assume that $R$ is an associative, unital algebra over $\mathbb{C}$. We denote the category of commutative superalgebras over $\mathbb{C}$ by $\mathbf{sAlg}_{\mathbb{C}}$ and equip its opposite category with the fpqc topology. Thus, a covering is induced by a jointly faithfully flat family of morphisms of superalgebras.

Recall the standard definition of the general linear supergroup $GL_{n|m}$, which is the functor that sends $R\in \mathbf{sAlg}_{\mathbb{C}}$ to the group of even automorphisms of the supermodule $R^{n|m}$. We use the standard description of it as the set of invertible supermatrices of the form
$$GL_{n|m}(R) := \left\{ \left( \begin{array}{c|c} x & \xi \\ \hline \eta & y \end{array}\right) \ \Big| \begin{array}{c} x \in GL_n(R_0), \ y \in GL_m(R_0), \\ \xi \in Mat_{n\times m }(R_1), \ \eta \in Mat_{m\times n }(R_1) \end{array} \right\}.$$
Supermatrices are added and multiplied just as ordinary matrices are. For a matrix of the displayed form, the inverse can be written using the usual Schur-complement formulas; the products of the odd off-diagonal blocks enter as even nilpotent corrections.

The superdeterminant or Berezinian $\operatorname{Ber} : GL_{n|m}(R) \rightarrow R_0^*$ is computed by the formula:
$$\operatorname{Ber}\left( \begin{array}{c|c} x & \xi \\ \hline \eta & y \end{array}\right) = \frac{\operatorname{det}(x - \xi y^{-1} \eta)}{\operatorname{det}(y)}.$$
The Berezinian satisfies similar properties to the ordinary determinant. In particular, it is multiplicative. 
The inverse of a supermatrix $A \in GL_{n|m}(R)$ can be computed by the super Cramer rule (see \cite{Shemyakova_2022}):
\begin{align}\label{sup_cram}
(A^{-1})_{ij} = 
\begin{cases}
\dfrac{(\operatorname{adj} A)_{ji} } {\operatorname{Ber} A}, & j = 1, \dots , n \\ \dfrac{(\operatorname{adj} A)_{ji}}{\operatorname{Ber}^* A}, & j = \widehat{1}, \dots , \widehat{m}
\end{cases}, 
\qquad
(\operatorname{adj} A)_{ij} = 
\begin{cases} \operatorname{Ber} D_{ij}(A), & i = 1, \dots , n \\ \operatorname{Ber}^* D_{ij}(A), & i = \widehat{1}, \dots , \widehat{m} \end{cases}.
\end{align}
We denote by $D_{ij}(A)$ the matrix obtained from $A$ by replacing its $i$-th row with the $j$-th basis vector.

With these notions in place, we define the affine Grassmannian for the supergroup $G=GL_{n|m}$. For $R\in\mathbf{sAlg}_{\mathbb{C}}$, set
\[
LG(R)=G(R((t))),\qquad L^+G(R)=G(R\llbracket t\rrbracket),
\]
where $R\llbracket t\rrbracket=R_0\llbracket t\rrbracket\oplus R_1\llbracket t\rrbracket$ and $R((t))=R_0((t))\oplus R_1((t))$.
We also write $G\llbracket t\rrbracket$ for the positive-loop group functor $L^+G$, so that $G\llbracket t\rrbracket(R)=G(R\llbracket t\rrbracket)$.

\begin{definition}\label{def:affine-super-grassmannian}
The affine super Grassmannian of $G$ is the fpqc sheaf
\[
Gr_G:=a_{\mathrm{fpqc}}(LG/L^+G),
\]
where $LG/L^+G$ denotes the presheaf quotient for the right-multiplication action and $a_{\mathrm{fpqc}}$ denotes fpqc sheafification. Thus, the displayed quotient on $R$-points need not be interpreted as a pointwise quotient before sheafification; this distinction is standard in the functor-of-points treatment of homogeneous superspaces \cite{balduzzi2009quotients}.
\end{definition}

For $G=GL_{n|m}$ this sheaf has a lattice description. Put $\Lambda_R=R\llbracket t\rrbracket^{n|m}$. A \emph{superlattice} in $R((t))^{n|m}$ is an $R\llbracket t\rrbracket$-submodule $L$ such that, fpqc-locally on $R$, it is free of rank $n|m$ and, for some $N\geq 0$,
\begin{equation}\label{bounded-lattice}
t^N\Lambda_R\subseteq L\subseteq t^{-N}\Lambda_R,
\qquad
L\otimes_{R\llbracket t\rrbracket}R((t))\cong R((t))^{n|m}.
\end{equation}
We additionally require $L/t^N\Lambda_R$ to be locally a direct summand of $t^{-N}\Lambda_R/t^N\Lambda_R$ as an $R$-supermodule; this condition is independent of enlarging $N$.

\begin{proposition}[Lattice model and ind-representability]\label{prop:lattice-model}
The map $g\mapsto g\Lambda_R$ identifies $Gr_G$ with the fpqc sheaf of superlattices. Moreover, $Gr_G$ is represented by an ind-superscheme
\[
Gr_G=\varinjlim_{N}Gr_G^{(N)},
\]
where $Gr_G^{(N)}$ is the subfunctor of lattices satisfying \eqref{bounded-lattice} for the fixed integer $N$.
\end{proposition}
\begin{pf}
The assignment $g\mapsto g\Lambda_R$ is invariant under right multiplication by $L^+G(R)$. Conversely, a superlattice admits an fpqc-local homogeneous basis; writing this basis as the columns of a matrix gives an element of $LG$, and two such bases differ by an element of $L^+G$. This proves the first assertion after fpqc sheafification.

For the second assertion, set
\[
M_{N,R}:=t^{-N}\Lambda_R/t^N\Lambda_R\cong R^{2Nn|2Nm}.
\]
A lattice in $Gr_G^{(N)}(R)$ determines the locally direct summand $L/t^N\Lambda_R\subset M_{N,R}$, which is stable under the nilpotent endomorphism induced by multiplication by $t$. Conversely, the inverse image of any such $t$-stable summand is a bounded superlattice. For each fixed superrank, locally direct summands are represented by a finite-dimensional super Grassmannian, and $t$-stability is the closed condition that the composite
\[
L/t^N\Lambda_R\longrightarrow M_{N,R}\xrightarrow{t}M_{N,R}\longrightarrow M_{N,R}/(L/t^N\Lambda_R)
\]
vanish. Hence $Gr_G^{(N)}$ is a finite disjoint union of closed subsuperschemes of super Grassmannians. The inclusions $Gr_G^{(N)}\hookrightarrow Gr_G^{(N+1)}$ are closed immersions, giving the claimed ind-superscheme structure. The super Grassmannians used here may also be described through super Pl\"ucker coordinates as in \cite{Shemyakova_2022}.
\end{pf}

\paragraph{Conventions on points, topologies and charts.}
We use the functor-of-points language throughout.  A \emph{superscheme} is a
functor on $\mathbf{sAlg}_{\mathbb C}$ that is a Zariski sheaf and is covered by open
subfunctors isomorphic to affine superspaces; the underlying topological space
of $\operatorname{Spec}R$ is $\operatorname{Spec}R_0$, so Zariski coverings are
generated by the sets $D(f)$, $f\in R_0$.  The two topologies play different
roles in this paper.  The fpqc topology is needed to define $Gr_G$ and the
cells $Gr^\lambda$ as sheaf quotients and sheaf images (Definition
\ref{def:affine-super-grassmannian}, Proposition \ref{prop:lattice-model},
Definition \ref{def:schubert-cell} and Proposition
\ref{prop:homogeneous-cell}): an $R$-point of $Gr^\lambda$ is, in general, of
the form $A[t^\lambda]$ only fpqc-locally on $R$, not for a single
$A\in L^+G(R)$.  The charts $U_0,U_1$ used in Sections \ref{nm11}
and \ref{nm21} are open subfunctors, defined by the invertibility of an even
element, and the gluing of charts is Zariski gluing; the morphisms between
charts are the transition maps.  No sheaf of sections over open subsets of
$\mathbb P^1$ is used.  For a finitely generated free $\mathbb C$-module $V$ we write
$\Pi V$ for the odd affine space with $(\Pi V)(R)=R_1\otimes V$; thus
$\Pi A_k=\mathbb A^{0|k}$ and $\Pi(A_k\oplus A_l)=\mathbb A^{0|k+l}$.

\paragraph{The coweight lattice.}
We write
\[
\mathcal{X}_{+}(n,m)=\left\{(\lambda_1,\ldots,\lambda_n,\lambda_{\hat 1},\ldots,\lambda_{\hat m})\in\mathbb{Z}^n\times\mathbb{Z}^m\ \middle|
\lambda_1\geq\cdots\geq\lambda_n,\
\lambda_{\hat 1}\geq\cdots\geq\lambda_{\hat m}\right\},
\]
and $t^\lambda=\operatorname{diag}(t^{\lambda_1},\ldots,t^{\lambda_n},t^{\lambda_{\hat 1}},\ldots,t^{\lambda_{\hat m}})$.
On complex points only the even part of the group matters:
$Gr_G(\mathbb C)=G_0(\mathbb C((t)))/G_0(\mathbb C\llbracket t\rrbracket)$ with
$G_0=GL_n\times GL_m$, since for $R=\mathbb C$ one has $R_1=0$.

\begin{proposition}[Cartan--Iwahori--Matsumoto decomposition]\label{Cartan}
There is a bijection
\begin{align*}
\mathcal{X}_{+}(n,m)& \overset{1:1}{\longleftrightarrow}
\left\{G_0(\mathbb C\llbracket t\rrbracket)\text{-orbits in }Gr_G(\mathbb{C})\right\},\\
\lambda & \longleftrightarrow G_0(\mathbb C\llbracket t\rrbracket)\cdot[t^{\lambda}].
\end{align*}
Equivalently, $\lambda\mapsto Gr^\lambda(\mathbb C)$ is a bijection from
$\mathcal X_+(n,m)$ onto the set of non-empty Schubert cells on $\mathbb C$-points.
\end{proposition}
\begin{pf}
This is the classical Cartan decomposition for $GL_n\times GL_m$ over
$\mathbb C((t))$: by the Smith normal form over the discrete valuation ring
$\mathbb C\llbracket t\rrbracket$, every element of $GL_n(\mathbb C((t)))$ can be
written as $X\,t^{\mu}\,Y$ with $X,Y\in GL_n(\mathbb C\llbracket t\rrbracket)$ and
$\mu_1\ge\dots\ge\mu_n$, and $\mu$ is unique, because it is determined by the
elementary divisors; the same holds for $GL_m$ (see \cite{zhu_2016_introduction,peerenboom2021affine}).
Applying this to both diagonal blocks proves the claim.
\end{pf}

\begin{rmk}[Odd tails]\label{rmk:tails}
Proposition \ref{Cartan} concerns $\mathbb C$-points, i.e.\ the reduced
ind-scheme; it does not say that every point of $Gr_G(R)$ is, even locally, in
the $L^+G(R)$-orbit of some $[t^\lambda]$, and for superalgebras $R$ with
non-trivial odd part this is false.  For instance, for $GL_{1|1}$ and
$R=\Lambda(\xi)$ the element
$A=\begin{pmatrix}1&\xi t^{-1}\\0&1\end{pmatrix}\in G(R((t)))$ defines a point of $Gr_G(R)$ whose body
is the base point, so that its orbit would have to be that of $[t^\lambda]$ with
$\lambda=0$, i.e.\ $A$ would have to lie in $L^+G(R)$, which it does not.
For $GL_{n|m}$ one takes $1+\xi t^{-1}E_{1\hat1}$, with the odd entry in
position $(1,\hat 1)$.  There are also even nilpotent tails, e.g.\
$\operatorname{diag}(1+\xi_1\xi_2t^{-1},1,\dots,1)$ for $R=\Lambda(\xi_1,\xi_2)$.
Consequently the Schubert cells $Gr^\lambda$ do not cover $Gr_G$: the union
$\bigcup_\lambda Gr^\lambda$ is the part of $Gr_G$ whose points are, fpqc-locally,
in the $L^+G$-orbits of the $[t^\lambda]$, and it exhausts the
complex points.  Throughout the paper the parameter $\lambda$ of a cell is the one
in Proposition \ref{Cartan}, and no statement about the cells depends on a
Cartan decomposition over superalgebras.  (In the first version of this paper,
the decomposition was stated for the Grassmann-envelope points
$G(\Gamma((t)))/G(\Gamma\llbracket t\rrbracket)$, where it fails as just shown.)
\end{rmk}

\begin{definition}\label{def:schubert-cell}
For $\lambda\in\mathcal{X}_{+}(n,m)$, the \emph{Schubert cell} $Gr^\lambda$ is the fpqc sheaf image of the orbit map
\[
L^+G\longrightarrow Gr_G,\qquad A\longmapsto A[t^\lambda].
\]
Its closure $\overline{Gr^\lambda}$ in the ind-superscheme $Gr_G$ is called the corresponding \emph{Schubert variety}.
\end{definition}

Let $H_\lambda$ be the stabilizer group functor of $[t^\lambda]$. Directly from the definition of the right quotient,
\begin{align}
H_\lambda
&=\left\{A\in L^+G\ \middle|\ [At^\lambda]=[t^\lambda]\right\}
=\left\{A\in L^+G\ \middle|\ A=t^\lambda Bt^{-\lambda}\text{ for some }B\in L^+G\right\}\notag\\
&=L^+G\cap t^\lambda L^+Gt^{-\lambda}.\label{stab_t_lambda}
\end{align}

\begin{proposition}[Homogeneous-space description]\label{prop:homogeneous-cell}
The orbit map induces an isomorphism of fpqc sheaves
\begin{equation}\label{mod_stab_t_lambda}
Gr^\lambda\cong L^+G/H_\lambda
=L^+G/(L^+G\cap t^\lambda L^+Gt^{-\lambda}).
\end{equation}
\end{proposition}
\begin{pf}
The kernel of the orbit map is $H_\lambda$. If $A[t^\lambda]=B[t^\lambda]$, then fpqc-locally there exists $C\in L^+G$ such that $At^\lambda=Bt^\lambda C$; equivalently, $B^{-1}A=t^\lambda Ct^{-\lambda}\in H_\lambda$. Thus, the fibers of the orbit map are precisely the right $H_\lambda$-cosets. Passing to the fpqc quotient identifies $L^+G/H_\lambda$ with the sheaf image that defines $Gr^\lambda$.
\end{pf}

Our computations in Theorems \ref{th11}, \ref{th21pos}, and \ref{th21mix} hold functorially over superalgebras. Equalities of quotient points are therefore understood fpqc-locally; the explicit normal forms constructed below give representatives on the stated charts.
\section{Results}\label{results}
\subsection{General description of Schubert cells}
The following theorem gives a general description of the Schubert cells.
\begin{theorem}[General Schubert cell]\label{gen}
Let $G=GL_{n|m}$ and let $\lambda=(\lambda_1,\ldots,\lambda_n|\lambda_{\hat 1},\ldots,\lambda_{\hat m})\in\mathcal{X}_{+}(n,m)$. Let $P_\lambda\subset G$ be the parabolic subsupergroup whose $(a,b)$-entry is zero whenever $\lambda_a>\lambda_b$. Then the Schubert cell $Gr^\lambda$ is a locally trivial super-affine bundle over the super flag variety $G/P_\lambda$, with fiber
\[
F_\lambda=ev_0^{-1}(P_\lambda)/H_\lambda,
\qquad ev_0:L^+G\longrightarrow G,\qquad A(t)\longmapsto A(0).
\]
If $I_{\bar 0}=\{1,\ldots,n\}$ and $I_{\bar 1}=\{\hat 1,\ldots,\hat m\}$, then
\[
F_\lambda\cong\mathbb{A}^{e_\lambda|o_\lambda},
\]
where
\begin{align*}
e_\lambda
&=\sum_{\substack{a,b\in I_{\bar 0}\\\lambda_a>\lambda_b}}
(\lambda_a-\lambda_b-1)
+\sum_{\substack{a,b\in I_{\bar 1}\\\lambda_a>\lambda_b}}
(\lambda_a-\lambda_b-1),\\
o_\lambda
&=\sum_{\substack{a\in I_{\bar 0},\,b\in I_{\bar 1}\\
\lambda_a>\lambda_b}}(\lambda_a-\lambda_b-1)
+\sum_{\substack{a\in I_{\bar 1},\,b\in I_{\bar 0}\\
\lambda_a>\lambda_b}}(\lambda_a-\lambda_b-1).
\end{align*}
In particular, $Gr^\lambda$ is a smooth finite-dimensional supervariety of superdimension
\[
\left(
\sum_{\lambda_i>\lambda_j}(\lambda_i-\lambda_j)
+\sum_{\lambda_{\hat i}>\lambda_{\hat j}}(\lambda_{\hat i}-\lambda_{\hat j})
\ \middle|
\sum_{\lambda_i>\lambda_{\hat j}}(\lambda_i-\lambda_{\hat j})
+\sum_{\lambda_{\hat i}>\lambda_j}(\lambda_{\hat i}-\lambda_j)
\right).
\]
\end{theorem}
\begin{pf}
The proof has three parts: we compute the stabilizer $H_\lambda$ (Step 1), identify the map to the flag variety and its fiber (Step 2), and show local triviality (Step 3) and that the fiber is an affine superspace (Step 4); Steps 5 and 6 count dimensions.

\noindent\emph{Step 1: the stabilizer.}
The classical argument underlying this proof is \cite[Theorem 2.3.15]{peerenboom2021affine}; we now give its functorial super-analogue. Let $R\in\mathbf{sAlg}_{\mathbb C}$, let $I=I_{\bar 0}\sqcup I_{\bar 1}$, and set
\[
d_{ab}:=\lambda_a-\lambda_b,\qquad a,b\in I.
\]
For $A(t)=(a_{ab}(t))\in L^+G(R)$, formula \eqref{stab_t_lambda} gives
\begin{equation}\label{stabilizer-entry-condition}
a_{ab}(t)\in
t^{\max\{0,d_{ab}\}}R_{|a|+|b|}\llbracket t\rrbracket
\qquad\text{for every }a,b.
\end{equation}
Indeed, the $(a,b)$-entry of $t^{-\lambda}A(t)t^\lambda$ is $t^{\lambda_b-\lambda_a}a_{ab}(t)$.

\smallskip
\noindent\emph{Step 2: evaluation at zero and the parabolic.}
If $d_{ab}>0$, condition \eqref{stabilizer-entry-condition} forces $a_{ab}(0)=0$; hence $ev_0(H_\lambda)\subseteq P_\lambda$. Conversely, every $p\in P_\lambda(R)$, regarded as a constant loop, satisfies \eqref{stabilizer-entry-condition}. Therefore
\[
ev_0(H_\lambda)=P_\lambda.
\]
Let
\[
E_\lambda:=ev_0^{-1}(P_\lambda)\subset L^+G.
\]
We have $H_\lambda\subseteq E_\lambda$, and evaluation induces
\[
\pi_\lambda:L^+G/H_\lambda\longrightarrow G/P_\lambda,
\qquad [A(t)]\longmapsto[A(0)].
\]
This is well defined because right multiplication by $H_\lambda$ changes $A(0)$ by right multiplication by $P_\lambda$. By Proposition \ref{prop:homogeneous-cell}, its source is $Gr^\lambda$, and its fiber over the identity coset is
\[
\pi_\lambda^{-1}([1])
=E_\lambda/H_\lambda
=F_\lambda.
\]

\smallskip
\noindent\emph{Step 3: local triviality.}
The super flag variety $G/P_\lambda$ has the standard big-cell cover. On the big cell, multiplication by the opposite unipotent subsupergroup gives a section of $G\to G/P_\lambda$; its translates give sections on a cover. The construction of super flag varieties and their big cells is described in \cite[Sections 3 and 5]{fioresi2007minkowski}; the general sheafified homogeneous-quotient formalism is given in \cite[Theorem 4.1 and Remark 4.2]{balduzzi2009quotients}.

Let $U\subseteq G/P_\lambda$ be one such chart and let $s:U\to G$ be a local section. We regard $s$ as taking values in $L^+G$ by the constant-loop embedding $G\hookrightarrow L^+G$. Then
\[
\Phi_U:U\times F_\lambda\longrightarrow\pi_\lambda^{-1}(U),
\qquad
(x,[e])\longmapsto[s(x)e],
\]
is an isomorphism. Indeed, if $[A(t)]\in\pi_\lambda^{-1}(U)$ and $x=[A(0)]$, then
\[
s(x)^{-1}A(t)\in E_\lambda,
\]
because its value at $t=0$ belongs to $P_\lambda$. Thus, the inverse map is
\[
[A(t)]\longmapsto
\bigl([A(0)],[s([A(0)])^{-1}A(t)]\bigr).
\]
Both maps are independent of the chosen representatives modulo $H_\lambda$. Hence $\pi_\lambda$ is locally trivial with fiber $F_\lambda$.

\smallskip
\noindent\emph{Step 4: coordinates on the fiber.}
Let $N:=\ker(ev_0)=\{A\in L^+G:\ A(0)=1\}$ and $\widetilde N:=N\cap H_\lambda$.

\smallskip
\noindent\emph{(a) Reduction to $N$.}
Every constant loop $p\in P_\lambda(R)$ lies in $H_\lambda(R)$ (Step 2).  If
$e\in E_\lambda(R)$, then $n:=e\,e(0)^{-1}\in N(R)$ and $[e]=[n]$ in
$E_\lambda/H_\lambda$, because $e(0)\in P_\lambda\subset H_\lambda$; and
$[n]=[n']$ for $n,n'\in N$ if and only if $n^{-1}n'\in N\cap H_\lambda=\widetilde N$.
Hence the inclusion $N\subset E_\lambda$ induces an isomorphism
$N/\widetilde N\cong E_\lambda/H_\lambda=F_\lambda$.

\smallskip
\noindent\emph{(b) A unipotent group of representatives.}
Let $\Sigma=\{(a,b):d_{ab}>0\}$.  For $(a,b)\in\Sigma$ the entry $a_{ab}$ of an
element of $E_\lambda$ is divisible by $t$, whereas for an element of
$H_\lambda$ it is divisible by $t^{d_{ab}}$ by
\eqref{stabilizer-entry-condition}; the entries at positions outside $\Sigma$
are not constrained beyond being power series.  Let $\mathcal U(R)$ be the set of
matrices
\[
u=1+\sum_{(a,b)\in\Sigma}\Bigl(\sum_{r=1}^{d_{ab}-1}c_{ab,r}\,t^r\Bigr)E_{ab},
\qquad c_{ab,r}\in R_{|a|+|b|}.
\]
The set $\Sigma$ is transitive: $d_{ab}>0$ and $d_{bc}>0$ imply $d_{ac}=d_{ab}+d_{bc}>0$,
and $(a,a)\notin\Sigma$.  Hence a product of two matrices supported on
$\Sigma$ is supported on $\Sigma$, and the $(a,c)$-entry of a product of two elements
of $\mathcal U$ has $t$-degrees between $2$ and $(d_{ab}-1)+(d_{bc}-1)\le d_{ac}-2$.
Thus $\mathcal U$ is closed under multiplication; it is also closed under
inversion, because $u-1$ is nilpotent.  The parities $|a|+|b|$ make
$\mathcal U$ a subgroup of $N$, and as a functor it is the affine superspace with
the coordinates $c_{ab,r}$.

\smallskip
\noindent\emph{(c) $\mathcal U\to N/\widetilde N$ is an isomorphism.}
\emph{Injectivity.}  If $[u]=[u']$ for $u,u'\in\mathcal U(R)$, then
$u^{-1}u'\in\mathcal U\cap\widetilde N$.  Its entries at positions $(a,b)\in\Sigma$
have degree $<d_{ab}$ (being elements of $\mathcal U$) and are divisible by
$t^{d_{ab}}$ (being elements of $H_\lambda$), hence vanish, so $u=u'$.

\emph{Surjectivity.}  Put $N^{(s)}=\{A\in N:\ A\equiv1\bmod t^s\}$ and let $s_0=\max d_{ab}$.
We show by descending induction on $s$ that $N^{(s)}\subseteq\mathcal U\cdot\widetilde N$.
For $s\ge s_0$ every element of $N^{(s)}$ already lies in $\widetilde N$.  Let $s<s_0$,
let $n\in N^{(s)}(R)$ and write $n\equiv1+t^sX\pmod{t^{s+1}}$ with
$X$ a constant matrix.  Let $X_<$ be the part of $X$ supported on the pairs
$(a,b)\in\Sigma$ with $d_{ab}>s$, and put $u_s:=1+t^sX_<\in\mathcal U$ and
$h_s:=1+t^s(X-X_<)$.  The entries of $X-X_<$ at positions of $\Sigma$ have
$d_{ab}\le s$, so $h_s\in\widetilde N$.  Then
$n_2:=u_s^{-1}\,n\,h_s^{-1}\equiv1\pmod{t^{s+1}}$, i.e.\ $n_2\in N^{(s+1)}$, and
by induction $n_2=u'h'$ with $u'\in\mathcal U$, $h'\in\widetilde N$.  Therefore
$n=(u_su')(h'h_s)\in\mathcal U\cdot\widetilde N$.  Taking $s=1$ gives $N=\mathcal U\widetilde N$.
(The factorization is taken with $\mathcal U$ on the left and $\widetilde N$ on the right,
which matches the right quotient; $\widetilde N$ is not normal in $N$, and this order is what makes
the induction work.)

\smallskip
All constructions are explicit matrix identities over $R$, hence natural in $R$,
and combined with (a) they give
\[
F_\lambda(R)\cong
\prod_{\substack{a,b\in I\\d_{ab}>0}}
\bigl(R_{|a|+|b|}\bigr)^{d_{ab}-1},
\qquad
F_\lambda\cong\mathbb A^{e_\lambda|o_\lambda}.
\]
This is the same coefficient count as in the classical calculation
\cite[Theorem 2.3.15]{peerenboom2021affine}, with the additional parity assignment $|a|+|b|$.

\smallskip
\noindent\emph{Step 5: dimension of the base.}
On the big cell of $G/P_\lambda$, there is one degree-zero coordinate $c_{ab,0}$ for each ordered pair satisfying $d_{ab}>0$. Its parity is again $|a|+|b|$. Hence
\begin{align*}
\dim_{\bar0}(G/P_\lambda)
&=\#\{(a,b)\in I_{\bar0}^2:d_{ab}>0\}
+\#\{(a,b)\in I_{\bar1}^2:d_{ab}>0\},\\
\dim_{\bar1}(G/P_\lambda)
&=\#\{(a,b)\in I_{\bar0}\times I_{\bar1}:d_{ab}>0\}\\
&\qquad
+\#\{(a,b)\in I_{\bar1}\times I_{\bar0}:d_{ab}>0\}.
\end{align*}
The corresponding classical root-space calculation is given in
\cite[Proposition 2.1.5]{zhu_2016_introduction}.

\smallskip
\noindent\emph{Step 6: total superdimension and smoothness.}
For every ordered pair with $d_{ab}>0$, the base contributes one coordinate and the fiber contributes $d_{ab}-1$ coordinates. Thus, the total contribution is $d_{ab}=\lambda_a-\lambda_b$. Separating pairs of equal and opposite parity gives exactly the superdimension displayed in the statement.

Finally, the standard big cells show that $G/P_\lambda$ is smooth, and $F_\lambda$ is a super-affine space. Since $Gr^\lambda$ is locally a product of these two smooth finite-dimensional supervarieties, it is smooth and finite-dimensional. This argument concerns the cell $Gr^\lambda$ and makes no smoothness claim about its closure $\overline{Gr^\lambda}$.
\end{pf}

In the following subsections, we analyze Schubert cells in greater detail, focusing on the cases of $GL_{1|1}$ and $GL_{2|1}$. Our computations follow this scheme: (1) we obtain a system of matrix representatives on explicit charts; and (2) we verify that, for any two representatives $A$ and $B$ in this system, $AB^{-1}$ belongs to the stabilizer $H_\lambda$ if and only if $A=B$. Thus, two representatives determine the same equivalence class if and only if they are equal.
\subsection{Schubert cells for $n|m = 1|1$}\label{nm11}
We begin with a brief remark on the notation used throughout our computations. Latin letters denote even matrix elements, and Greek letters denote odd matrix elements. If a supermatrix is denoted by a capital Latin letter, then its even blocks are denoted by the corresponding lowercase letter. For example, the matrix $X$ has the blocks $x$ and $x'$ on its diagonal.

The following theorem gives an explicit coordinate description of Schubert cells in dimension $1|1$.
\begin{theorem}\label{th11}
For the affine Grassmannian of the supergroup $GL_{1|1}$, the Schubert cells with cocharacter $\lambda=(k|0)$ (or $\lambda = (0|k)$), where $k>0$, are parametrized by $k$ odd coordinates; that is,
$$Gr^{(0|k)}_{GL_{1|1}} \cong Gr^{(k|0)}_{GL_{1|1}} \cong \mathbb{A}^{0|k},$$
where $\mathbb A^{0|k}=\Pi A_k$ is the odd affine superspace with $R$-points $R_1^{\,k}$ (in the notation of Section \ref{jets}, $\mathbb A^{0|k}=J_{k-1}(\mathbb P^{0|1})$).
\end{theorem}
\begin{pf}
Fix $R\in\mathbf{sAlg}_{\mathbb{C}}$ and let $\lambda = (k|0)$, $k>0$. For an $R$-point of the orbit $Gr^\lambda$ we have:
\begin{align*}
X\cdot [t^{\lambda}] =\left(\begin{array}{c|c} x & \kappa  \\ \hline \kappa' & x' \end{array}\right)\cdot\Biggl[\left(\begin{array}{c|c} t^k & 0  \\ \hline 0 & 1 \end{array}\right)\Biggr] &= \Biggl[\left(\begin{array}{c|c} xt^k & \kappa  \\ \hline \kappa' t^k & x' \end{array}\right)\cdot\left(\begin{array}{c|c} x' & 0  \\ \hline -\kappa't^k & 1 \end{array}\right)\Biggr] \\
&=\Biggl[\left(\begin{array}{c|c} (xx' -\kappa\kappa')t^k & \kappa  \\ \hline 0 & x' \end{array}\right)\cdot\left(\begin{array}{c|c} (xx' -\kappa\kappa')^{-1} & 0  \\ \hline 0 & x'^{-1} \end{array}\right)\Biggr] \\
&=\Biggl[\left(\begin{array}{c|c} t^k &  \dfrac{\kappa}{x'}  \\ \hline 0 & 1 \end{array}\right)\cdot\left(\begin{array}{c|c} 1& y \\ \hline 0 & 1 \end{array}\right)\Biggr] = \Biggl[\left(\begin{array}{c|c} t^k &  \widetilde{\kappa} \\ \hline 0 & 1 \end{array}\right) \Biggr] \quad  \\
&= \left(\begin{array}{c|c} 1 &  \widetilde{\kappa} \\ \hline 0 & 1 \end{array}\right) \Biggl[\left(\begin{array}{c|c} t^k &  0 \\ \hline 0 & 1 \end{array}\right) \Biggr], \text{ where } \operatorname{deg}_t \widetilde{\kappa} < k.
\end{align*}
In the third step, $\widetilde\kappa:=[\kappa/x']_{<k}$ and $y:=-t^{-k}[\kappa/x']_{\ge k}\in R_1\llbracket t\rrbracket$, so that $t^ky+\kappa/x'=\widetilde\kappa$.
 All the matrices used belong to $GL_{1|1}\llbracket t \rrbracket$, since $x$ and $x'$ are invertible elements of $R_0\llbracket t\rrbracket$ by the invertibility criterion for supermatrices.
 Next, we choose two representatives $A, B \in GL_{1|1}\llbracket t \rrbracket$ of the form obtained above; that is, $A = \left(\begin{array}{c|c} 
 1 & \widetilde{\alpha}  \\ \hline  0 & 1\end{array}\right)$ and $B = \left(
\begin{array}{c|c} 
 1 & \widetilde{\beta}  \\ \hline 
0 & 1
\end{array}\right)$.
 We then use \eqref{sup_cram} to invert $B$ and confirm that the matrices above indeed form a complete system of representatives. Recall the description of the stabilizer in \eqref{stab_t_lambda}. We now write down condition \eqref{mod_stab_t_lambda}, which states that $A$ and $B$ belong to the same class modulo the stabilizer if and only if there exists $C \in GL_{1|1}\llbracket t \rrbracket$ such that
\begin{align*}
\left(
\begin{array}{c|c} 
 c & \gamma t^{k}  \\ \hline  
 \gamma' t^{ -k} & c'
\end{array}\right)
&=
\left(
\begin{array}{c|c} 
1 & - \widetilde{\beta} \\ \hline 
0 & 1 \end{array}
\right)
\cdot 
\left(
\begin{array}{c|c} 
 1 & \widetilde{\alpha}  \\ \hline  
0 & 1
\end{array}\right) = \left(
\begin{array}{c|c} 
 1 & \widetilde{\alpha} -  \widetilde{\beta}  \\ \hline  
0 & 1
\end{array}\right) \Longleftrightarrow \widetilde{\alpha} =  \widetilde{\beta}.
\end{align*}
Thus the matrices $\left(\begin{smallmatrix}1&\widetilde\kappa\\0&1\end{smallmatrix}\right)$ with $\deg_t\widetilde\kappa<k$ are unique representatives, so the odd series $\widetilde\kappa$ with $k$ odd coefficients is a coordinate on the cell. The isomorphism is given by
$$Gr^{\lambda}\ni X\cdot [t^{\lambda}] = \left(\begin{array}{c|c} x & \kappa  \\ \hline \kappa ' & x' \end{array} \right)\cdot \left[\left(\begin{array}{c|c} t^k & 0 \\ \hline 0 & 1 \end{array} \right)\right] \longmapsto (\left[ \kappa  / x' \right]_{0} , \dots , \left[ \kappa  / x ' \right]_{k-1}) \in \mathbb{A}^{0|k}(R)=R_1^{\,k},$$
where by $[-]_i$ we denote the $i$-th coefficient of the series.\\
The case $\lambda = (0|k)$ is analogous (alternatively, it follows from the case $\lambda=(k|0)$ by the isomorphism $GL_{1|1}\cong GL_{1|1}$ that exchanges the two basis vectors, i.e.\ shifts parity).
\end{pf}
\subsection{Schubert cells for $n|m = 2|1$}\label{nm21}
In this subsection, $G$ stands for $GL_{2|1}$, $Gr$ denotes the affine super Grassmannian functor $Gr_G$, and $R\in\mathbf{sAlg}_{\mathbb{C}}$ is arbitrary.

Here we perform computations similar to those in the previous subsection. We write relation \eqref{mod_stab_t_lambda} on $R$-points using matrices $A$, $B$, and $C$ of the form $A =\left( \begin{array}{c|c}a&\alpha \\ \hline \alpha' & a' \end{array}\right)$, where $a = \left( \begin{array}{cc}a_{11}&a_{12}\\ a_{21} & a_{22} \end{array}\right) \in GL_2(R_0\llbracket t\rrbracket)$, $a' \in R_0\llbracket t\rrbracket^*$, and $\alpha = (\alpha_1, \alpha_2)$ and $\alpha' = (\alpha_1', \alpha_2')$, with $\alpha_i,\alpha_i' \in R_1\llbracket t \rrbracket$ for $i = 1, 2$; these are anticommuting series.

We consider normalized dominant coweights of the form $\lambda=(k,l|0)$,
with $k\geq l$. The dominant parameter region has three open chambers: the
\emph{positive chamber} $k>l>0$, the \emph{mixed-sign chamber} $k>0>l$,
and the \emph{negative chamber} $0>k>l$. The two coordinate theorems below
treat the positive and mixed-sign chambers. The negative chamber is included
in the geometric summary because its dimensions follow directly from Theorem
\ref{gen}; equivalently, it is exchanged with the positive chamber by the
the automorphism $g\mapsto(g^{\mathrm{st}})^{-1}$ followed by the nontrivial even Weyl-group
element (Proposition \ref{prop:jet-cells}(iii)). This terminology refers only to the signs of the entries of
$\lambda$.
\begin{figure}[hp]
\centering
\begin{adjustbox}{max width=\linewidth}
\begin{tikzpicture}[
  scale=.92,
  chamber/.style={align=center,font=\small},
  boundary/.style={align=center,font=\scriptsize,text=red!70!black}]
  \def\xmin{-4}\def\xmax{5}\def\ymin{-4}\def\ymax{4}
  \fill[blue!16] (0,0)--(\xmax,0)--(\xmax,\ymax)--(\ymax,\ymax)--cycle;
  \fill[orange!20] (0,0)--(\xmax,0)--(\xmax,\ymin)--(0,\ymin)--cycle;
  \fill[purple!16] (\xmin,\ymin)--(0,0)--(0,\ymin)--cycle;
  \fill[gray!12,pattern=north east lines,pattern color=gray!45]
       (\xmin,\xmin)--(\xmin,\ymax)--(\ymax,\ymax)--cycle;
  \draw[->,thick] (\xmin-.25,0)--(\xmax+.4,0) node[right] {$k$};
  \draw[->,thick] (0,\ymin-.25)--(0,\ymax+.4) node[above] {$l$};
  \draw[thick,gray!75!black] (\xmin,\xmin)--(\ymax,\ymax);
  \node[rotate=45,above,font=\small] at (3.2,3.2) {$k=l$};
  \node[rotate=45,text=gray!65!black,font=\scriptsize] at (-2,1.65)
       {non-dominant region};
  \draw[red!70!black,dashed,very thick] (0,0)--(\xmax,0);
  \draw[red!70!black,dashed,very thick] (0,\ymin)--(0,0);
  \node[boundary,above] at (3.7,.08) {$l=0$: the mark $D_2$ disappears};
  \node[boundary,rotate=90,anchor=north] at (0.1,-2.3) {$k=0$: the mark $D_1$ disappears};
  \node[chamber,text=blue!60!black] at (3.6,1.55)
       {positive chamber, $k>l>0$\\ marks $\bigl((k,+),(l,+)\bigr)$\\ $\dim=(k-l\mid k+l)$};
  \node[chamber,text=orange!75!black] at (3.6,-2)
       {mixed-sign chamber, $k>0>l$\\ marks $\bigl((k,+),(-l,-)\bigr)$\\ $\dim=(k-l\mid k-l)$};
  \node[chamber,text=purple!70!black,font=\scriptsize] (neg) at (-2.2,-5.0)
       {negative chamber, $0>k>l$\\ marks $\bigl((-k,-),(-l,-)\bigr)$\\ $\dim=(k-l\mid -k-l)$};
  \draw[->,purple!70!black] (neg.north east)--(-1.0,-3.2);
  \node[align=center,font=\scriptsize,fill=white,inner sep=1.5pt]
       (diag) at (-3.5,0.0)
       {$k=l$: the reduced orbit is a point;\\ the two marks form one $GL_2$-multiplet};
  \draw[->,gray!70!black] (diag.south east)--(-0.35,-0.35);
\end{tikzpicture}
\end{adjustbox}
\caption{Chambers of normalized dominant coweights $\lambda=(k,l|0)$, labelled
by the decoration (Definition~\ref{def:decoration}) and the superdimension of
$Gr^\lambda$. A mark $(N,+)$ means $N$ odd coordinates in the upper-right
block, $(N,-)$ in the lower-left block. The dashed rays are the odd boundary loci,
the hatched half-plane is the non-dominant region.}
\label{fig:coweight-chambers}
\end{figure}

Figure \ref{fig:coweight-chambers} summarizes the coefficient count obtained
from Theorem \ref{gen}. Throughout the dominant region the reduced $GL_2$
orbit depends only on $d=k-l$: its even dimension is $d$, and it degenerates
to a point on the diagonal $d=0$. The two odd multiplicities are
\[
N_1=|k|,\qquad N_2=|l|,
\]
so that
\begin{equation}\label{eq:gl21-chamber-dimension}
\dim Gr^{(k,l|0)}=(k-l\mid |k|+|l|).
\end{equation}
More explicitly, the three open chambers have the following geometry:
\[
\begin{array}{c|c|c|c}
\text{chamber}&(N_1,N_2)&\dim Gr^{(k,l|0)}&
\text{location of the odd coordinates}\\ \hline
k>l>0&(k,l)&(k-l\mid k+l)&
\text{both in the upper-right block}\\
k>0>l&(k,-l)&(k-l\mid k-l)&
\text{one in each off-diagonal block}\\
0>k>l&(-k,-l)&(k-l\mid -k-l)&
\text{both in the lower-left block}
\end{array}
\]
Crossing $l=0$ changes the second family from upper-right to lower-left,
and crossing $k=0$ changes the first family in the same way; the relevant
family has zero multiplicity on the wall itself. Consequently, in the
mixed-sign chamber the odd dimension equals the even dimension $k-l$ and is
constant along every line $k-l=d$. In the positive and negative chambers the
even geometry still depends only on $k-l$, whereas the odd dimension also
records the position relative to the origin. On the diagonal $k=l\ne0$, the
reduced orbit has superdimension $\left(0\mid 2|k|\right)$; at $(0,0|0)$ both the
even and odd directions disappear.

Given an orbit point $A\cdot[t^\lambda]\in Gr^\lambda\subset Gr$, which maps to $[A]\in G\llbracket t\rrbracket/(t^\lambda G\llbracket t\rrbracket t^{-\lambda}\cap G\llbracket t\rrbracket)$, we can perform a transformation that yields a complete system of representatives.

We begin with a lemma describing a transformation applicable in both chambers.
\begin{lemma}[Point normalization]\label{point_norm}
Let $A=\left( \begin{array}{c|c}a&\alpha \\ \hline \alpha' & a' \end{array}\right) \in G\llbracket t \rrbracket$ be a matrix as above. Then
\begin{itemize}
\item[(1)] If $\lambda=(k,l|0)$ lies in the closed positive chamber $k\geq l\geq0$, then there exists a matrix of the form $\widetilde{A} = \left( \begin{array}{c|c}\widetilde{a}&\widetilde{\alpha} \\ \hline 0 \ 0& 1 \end{array}\right) \in G\llbracket t \rrbracket$ such that $A[t^\lambda] = \widetilde{A}[t^\lambda]$ in $Gr$.
\item[(2)] If $\lambda=(k,l|0)$ lies in the mixed-sign chamber $k>0>l$, then there exists a matrix of the form $\widetilde{A} = \left( \begin{array}{c|c}\widetilde{a}&\widetilde{\alpha} \\ \hline 0 \ \widetilde{\alpha}'& 1 \end{array}\right) \in G\llbracket t \rrbracket$ such that $A[t^\lambda] = \widetilde{A}[t^\lambda]$ in $Gr$.
\end{itemize}
\end{lemma}
\begin{pf}
\begin{itemize}
\item[(1)] We transform the general matrix by multiplying it on the right by matrices from $G\llbracket t \rrbracket$, which preserves the point $A[t^\lambda]$ in the quotient. The transformation is as follows:
\begin{align*}
A[t^\lambda] = \left[\left( 
 \arraycolsep=2pt\def\arraystretch{1.5}
 \begin{array}{cc|c}
a_{11}&a_{12}&\alpha_1  \\ 
a_{21}&a_{22}&\alpha_2  \\ \hline
\alpha_1' &\alpha_2'&a'
\end{array}
\right)\cdot
\left( 
 \arraycolsep=2pt\def\arraystretch{1.5}
 \begin{array}{cc|c}
t^k&0&0  \\ 
0&t^l&0  \\ \hline
0&0&1
\end{array}
\right)
 \right] &= 
\left[\left( 
 \arraycolsep=5pt\def\arraystretch{2.2}
 \begin{array}{cc|c}
a_{11}t^{k}&a_{12}t^{l}&\alpha_1  \\ 
a_{21}t^{k}&a_{22}t^{l}&\alpha_2  \\ \hline
\alpha_1' t^{k}&\alpha_2't^{l}&a'
\end{array}
\right) \cdot
\left( 
 \arraycolsep=2pt\def\arraystretch{1.5}
 \begin{array}{cc|c}
1&0&0  \\ 
0&1&0  \\ \hline
-\dfrac{\alpha_1't^{k}}{a'}&-\dfrac{\alpha_2't^{l}}{a'}&1/a'
\end{array}
\right)
\right] \\
&= \left[\left( 
 \arraycolsep=2pt\def\arraystretch{2}
 \begin{array}{cc|c}
(a_{11}-\dfrac{\alpha_1\alpha_1'}{a'})t^{k}&(a_{12}-\dfrac{\alpha_1\alpha_2'}{a'})t^{l}&\alpha_1/a'  \\ 
(a_{21}-\dfrac{\alpha_2\alpha_1'}{a'})t^{k}&(a_{22}-\dfrac{\alpha_2\alpha_2'}{a'})t^{l}&\alpha_2/a' \\ \hline
0&0&1
\end{array}
\right)
\right]\\
&=
\left[\left( 
 \arraycolsep=2pt\def\arraystretch{1.5}
 \begin{array}{cc|c}
\widetilde{a}_{11}&\widetilde{a}_{12}&\widetilde{\alpha}_1 \\ 
\widetilde{a}_{21}&\widetilde{a}_{22}&\widetilde{\alpha}_2  \\ \hline
0&0&1
\end{array}
\right)
\left( 
 \arraycolsep=2pt\def\arraystretch{1.5}
 \begin{array}{cc|c}
t^{k}&0&0  \\ 
0&t^l&0  \\ \hline
0&0&1
\end{array}
\right)
\right] = \widetilde{A}[t^\lambda].
\end{align*}
This gives the desired form of $\widetilde{A}$.
\item[(2)] In the mixed-sign chamber, the transformation matrix from (1) need not belong to $G\llbracket t \rrbracket$. Instead, one replaces its second column with the even basis vector $e_2=(0,1|0)$; the desired result then follows.
\end{itemize}
\end{pf}
For $r>0$, put $A_r:=\mathbb C[t]/(t^r)$.
Write $U_0,U_1\subset\mathbb P^1$ for the standard affine charts, with
coordinates $z_0,z_1$ satisfying $z_0z_1=1$ on $U_{01}:=U_0\cap U_1$.
We denote by $J_{r-1}(U_i)$ the corresponding $(r-1)$-jet chart: its
$R$-points are the truncated series
$\widetilde z_i\in R_0[t]/(t^r)$ whose constant term lies in $U_i(R_0)$.
On $J_{r-1}(U_{01})$, the series $\widetilde z_i$ is invertible.

\begin{rmk}[The charts $U_0$ and $U_1$ are well defined]\label{rmk:charts}
Let $x=[At^\lambda]\in Gr^\lambda(R)$ and let $A'=Ah$ with $h\in H_\lambda(R)$ be
another representative of the shape in Lemma \ref{point_norm}.  Since
$h(0)\in P_\lambda$ has vanishing entries at the positions $(a,b)$ with
$\lambda_a>\lambda_b$ and invertible diagonal blocks, one finds
$a'_{12}(0)\equiv a_{12}(0)\,h_{22}(0)$ and $a'_{22}(0)\equiv a_{22}(0)\,h_{22}(0)$
modulo the ideal of $R_0$ generated by products of two odd elements, a nil
ideal, with $h_{22}(0)$ a unit of $R_0$.  An element of $R_0$ is a unit if and
only if it is a unit modulo this ideal, so the conditions ``$a_{12}(0)$ is a
unit'' and ``$a_{22}(0)$ is a unit'' do not depend on the representative.
These are open conditions, so $U_0$ and $U_1$ are open subfunctors of
$Gr^\lambda$, which is the sense in which they are charts.
\end{rmk}

\begin{theorem}[Positive-chamber Schubert cell]\label{th21pos}
Let $\lambda=(k,l|0)$ with $k>l>0$, and set $d=k-l$.  Then $Gr^\lambda$
is represented by a superscheme $\mathcal P_{k,l}$ with two charts
\[
\begin{aligned}
\mathcal U_0&=J_{d-1}(U_0)\times\Pi(A_k\oplus A_l),\\
\mathcal U_1&=J_{d-1}(U_1)\times\Pi(A_l\oplus A_k).
\end{aligned}
\]
Let $\pi_l^k:A_k\to A_l$ be truncation and
$\iota_l^k:A_l\hookrightarrow A_k$ the standard inclusion.  On the overlap,
the gluing map $\Phi_{01}:\mathcal U_0|_{U_{01}}\to
\mathcal U_1|_{U_{01}}$ is
\begin{equation}\label{pos_rel}
\boxed{
\begin{aligned}
\widetilde z_1&=\widetilde z_0^{-1}\pmod {t^d},\\[2pt]
\begin{pmatrix}\xi_1\\ \eta_1\end{pmatrix}
&=
\underbrace{
\begin{pmatrix}
\pi_l^k&0\\
\widetilde z_1(\iota_l^k\pi_l^k-\mathrm{id}_{A_k})&\iota_l^k
\end{pmatrix}}_{\displaystyle \varphi_{01}}
\begin{pmatrix}\xi_0\\ \eta_0\end{pmatrix}.
\end{aligned}}
\end{equation}
In particular,
\[
(Gr^\lambda)_{\mathrm{red}}\cong J_{d-1}(\mathbb P^1),
\qquad
\dim Gr^\lambda=(d\mid k+l).
\]
\end{theorem}
\begin{pf}
We choose a representative $A\cdot[t^\lambda]\in Gr^\lambda$ as described in
part (1) of Lemma \ref{point_norm}.  The two charts are
$U_0=\{a_{22}\text{ is invertible}\}$ and
$U_1=\{a_{12}\text{ is invertible}\}$.

\smallskip
\noindent\emph{Step 1: normal form on $U_0$.}
Restricting to $U_0$, one can transform the representative as follows:
\begin{align*}
A\cdot [t^\lambda]=
\left[\left( 
 \arraycolsep=2pt\def\arraystretch{1.5}
 \begin{array}{cc|c}
a_{11}&a_{12}&\alpha_1  \\ 
a_{21}&a_{22}&\alpha_2  \\ \hline
0&0&1
\end{array}
\right)\cdot
\left( 
 \arraycolsep=2pt\def\arraystretch{1.5}
 \begin{array}{cc|c}
t^{k}&0&0  \\ 
0&t^{l}&0  \\ \hline
0&0&1
\end{array}
\right)
 \right] &=
 \left[\left( 
 \arraycolsep=2pt\def\arraystretch{1.5}
 \begin{array}{cc|c}
a_{11}t^k&a_{12}t^l&\alpha_1  \\ 
a_{21}t^k&a_{22}t^l&\alpha_2  \\ \hline
0&0&1
\end{array}
\right)\cdot
\left( 
 \arraycolsep=2pt\def\arraystretch{1.5}
 \begin{array}{cc|c}
a_{22}\operatorname{det}(a)^{-1}&0&0  \\ 
-a_{21}\operatorname{det}(a)^{-1}t^{k-l}&a_{22}^{-1}&0  \\ \hline
0&0&1
\end{array}
\right)
\right] \\
&
=\left[\left(
 \arraycolsep=2pt\def\arraystretch{1.5}
 \begin{array}{cc|c}
t^k&\dfrac{a_{12}}{a_{22}}t^l&\alpha_1  \\ 
0&t^l&\alpha_2  \\ \hline
0&0&1
\end{array}
\right)\cdot
\left( 
 \arraycolsep=2pt\def\arraystretch{1.5}
 \begin{array}{cc|c}
1&x&0  \\ 
0&1&0  \\ \hline
0&0&1
\end{array}
\right)
\right]
=\left[\left( 
 \arraycolsep=2pt\def\arraystretch{1.5}
 \begin{array}{cc|c}
t^k&a_0t^l&\alpha_1  \\ 
0&t^l&\alpha_2  \\ \hline
0&0&1
\end{array}
\right)\right],
\end{align*}
where we can choose $x = - \sum_{j\ge k-l}\left(\dfrac{a_{12}}{a_{22}}\right)^{(j)}t^{j-k+l}$ so that $a_0(t) = \sum_{j<k-l}\left(\dfrac{a_{12}}{a_{22}}\right)^{(j)}t^{j}$.  \\
We continue the transformation as follows:
\begin{align*}
\left[\left( 
 \arraycolsep=2pt\def\arraystretch{1.5}
 \begin{array}{cc|c}
t^k&a_0t^l&\alpha_1  \\ 
0&t^l&\alpha_2  \\ \hline
0&0&1
\end{array}
\right)\cdot
\left( 
 \arraycolsep=2pt\def\arraystretch{1.5}
 \begin{array}{cc|c}
1&0&x_1  \\ 
0&1&x_2 \\ \hline
0&0&1
\end{array}
\right)
\right]=
\left[\left( 
 \arraycolsep=2pt\def\arraystretch{1.5}
 \begin{array}{cc|c}
t^k&a_0t^l&\alpha_{1,0}  \\ 
0&t^l&\alpha_{2,0}  \\ \hline
0&0&1
\end{array}
\right)
\right]=
\left[\left( 
 \arraycolsep=2pt\def\arraystretch{1.5}
 \begin{array}{cc|c}
1&a_0&\alpha_{1,0}  \\ 
0&1&\alpha_{2,0}  \\ \hline
0&0&1
\end{array}
\right)
\left( 
 \arraycolsep=2pt\def\arraystretch{1.5}
 \begin{array}{cc|c}
t^k&0&0  \\ 
0&t^l&0  \\ \hline
0&0&1
\end{array}
\right)
\right]
,
\end{align*}
where $x_1$ and $x_2$ can be chosen so that
\begin{equation}\label{eq:positive-U0-coordinates}
\begin{aligned}
\deg_t a_0&<d,
&a_0&\equiv \frac{a_{12}}{a_{22}}&&\pmod {t^d},\\
\deg_t\alpha_{1,0}&<k,
&\alpha_{1,0}&\equiv
\alpha_1+a_0(\alpha_{2,0}-\alpha_2)&&\pmod {t^k},\\
\deg_t\alpha_{2,0}&<l,
&\alpha_{2,0}&\equiv\alpha_2&&\pmod {t^l}.
\end{aligned}
\end{equation}
\smallskip
\noindent\emph{Step 2: uniqueness of the normal form.}
We now rewrite relation \eqref{mod_stab_t_lambda} in dimension $2|1$. One can replace the inverse matrix $B^{-1}$ with the adjugate matrix, computed using the super Cramer rule \eqref{sup_cram}, because they differ by the diagonal matrix $\operatorname{diag} (\operatorname{Ber}(B), \operatorname{Ber}(B), \operatorname{Ber}^*(B))$, which commutes with $t^\lambda$ and thus lies in the same class modulo the stabilizer.
\begin{align}\label{slay_condition}
\left( 
 \arraycolsep=2pt\def\arraystretch{1.5}
 \begin{array}{cc|c}
c_{11}&c_{12}t^{k-l}&\gamma_1 t^{k} \\ 
c_{21}t^{l-k}&c_{22}&\gamma_2 t^{l} \\ \hline
0&0&1
\end{array} \right)
= 
\left( 
 \arraycolsep=2pt\def\arraystretch{1.5}
 \begin{array}{cc|c}
1&-b_{0}&-\beta_{1,0}+b_{0}\beta_2\\ 
0&1&-\beta_{2,0} \\ \hline
0&0&1
\end{array} \right)
\left( 
 \arraycolsep=2pt\def\arraystretch{1.5}
 \begin{array}{cc|c}
1&a_0&\alpha_{1,0}  \\ 
0&1&\alpha_{2,0}  \\ \hline
0&0&1
\end{array} \right).
\end{align}
 From \eqref{slay_condition} we see that the matrices
 \begin{align*}
 A[t^\lambda] = \left[\left( 
 \arraycolsep=2pt\def\arraystretch{1.5}
 \begin{array}{cc|c}
1&a_0&\alpha_{1,0}  \\ 
0&1&\alpha_{2,0}  \\ \hline
0&0&1
\end{array}
\right)
\left( 
 \arraycolsep=2pt\def\arraystretch{1.5}
 \begin{array}{cc|c}
t^k&0&0  \\ 
0&t^l&0  \\ \hline
0&0&1
\end{array}
\right)
\right]
 \end{align*}
 form a complete system of representatives on $U_0$, because
\begin{align*}
\begin{cases}
c_{12}t^{k-l} &= a_0 - b_0,\\
\gamma_1t^{k} &= \alpha_{1,0} - \beta_{1,0}-b_0(\alpha_{2,0}-\beta_{2,0}),\\
\gamma_2t^{l} &= \alpha_{2,0} - \beta_{2,0},
\end{cases}
\quad \Longleftrightarrow \quad
\begin{cases}
a_0 &= b_0,\\
\alpha_{1,0} &= \beta_{1,0},\\
\alpha_{2,0} &= \beta_{2,0}.
\end{cases}
\end{align*}
\smallskip
\noindent\emph{Step 3: the chart $U_1$.}
Let $P\in G_0(\mathbb C)\subset G\llbracket t\rrbracket$ be the permutation
matrix exchanging $e_1$ and $e_2$ (and fixing the odd basis vector).  Left
multiplication by $P$ is an automorphism of $Gr^\lambda$ with $P^2=1$; it
exchanges the first two rows of $A$, so $PA$ again has the shape of part (1) of
Lemma \ref{point_norm}, with $(PA)_{22}=a_{12}$ and with $\alpha_1,\alpha_2$
interchanged.  Hence $x\in U_1$ if and only if $Px\in U_0$.  Applying Steps 1 and
2 to $Px$ we get unique $(a,\alpha_1,\alpha_2)$ with
$[PA\,t^\lambda]=[N(a,\alpha_1,\alpha_2)\,t^\lambda]$, and therefore
$[A\,t^\lambda]=[P\,N(a,\alpha_1,\alpha_2)\,t^\lambda]$, where
\[
N(a_0,\alpha_{1,0},\alpha_{2,0})=
\left(
 \arraycolsep=2pt\def\arraystretch{1.2}
 \begin{array}{cc|c}
1&a_0&\alpha_{1,0}\\
0&1&\alpha_{2,0}\\ \hline
0&0&1
\end{array}\right),
\qquad
M(a_1,\alpha_{1,1},\alpha_{2,1})=
\left(
 \arraycolsep=2pt\def\arraystretch{1.2}
 \begin{array}{cc|c}
0&1&\alpha_{1,1}\\
1&a_1&\alpha_{2,1}\\ \hline
0&0&1
\end{array}\right).
\]
Indeed $PN(a,\alpha_1,\alpha_2)=M(a,\alpha_2,\alpha_1)$.  Thus the matrices
$M(a_1,\alpha_{1,1},\alpha_{2,1})\,t^\lambda$ with $\deg_ta_1<d$,
$\deg_t\alpha_{1,1}<l$ and $\deg_t\alpha_{2,1}<k$ form a complete system of
unique representatives on $U_1$; we take $(\widetilde z_1,\xi_1,\eta_1)=
(a_1,\alpha_{1,1},\alpha_{2,1})$ as coordinates on $\mathcal U_1$.  This replaces
the phrase ``applying similar transformations''; the explicit expressions of
these coordinates in terms of the entries of $A$ are obtained from
\eqref{eq:positive-U0-coordinates} by the substitution induced by $P$.

\smallskip
\noindent\emph{Step 4: the transition map and gluing.}
Let $x\in U_0$ have coordinates $(a_0,\alpha_{1,0},\alpha_{2,0})$.  Since
$a_0\equiv a_{12}/a_{22}\pmod{t^d}$ with $d>0$, we have
$x\in U_1$ if and only if the constant term $a_0(0)$ is a unit of $R_0$
(the conditions defining $U_0,U_1$ are independent of the representative, see
Remark \ref{rmk:charts}).  On this overlap put
\[
a_1:=[a_0^{-1}]_{<d},\qquad
\alpha_{1,1}:=[\alpha_{1,0}]_{<l},\qquad
\alpha_{2,1}:=\bigl[\alpha_{2,0}-a_1[\alpha_{1,0}]_{\ge l}\bigr]_{<k}.
\]
For $N=N(a_0,\alpha_{1,0},\alpha_{2,0})$ and $M=M(a_1,\alpha_{1,1},\alpha_{2,1})$ one
computes
\[
M^{-1}=\left(
 \arraycolsep=2pt\def\arraystretch{1.2}
 \begin{array}{cc|c}
-a_1&1&a_1\alpha_{1,1}-\alpha_{2,1}\\
1&0&-\alpha_{1,1}\\ \hline
0&0&1
\end{array}\right),
\qquad
M^{-1}N=\left(
 \arraycolsep=2pt\def\arraystretch{1.4}
 \begin{array}{cc|c}
-a_1&1-a_0a_1&B_{13}\\
1&a_0&\alpha_{1,0}-\alpha_{1,1}\\ \hline
0&0&1
\end{array}\right),
\]
with $B_{13}=\alpha_{2,0}-\alpha_{2,1}-a_1(\alpha_{1,0}-\alpha_{1,1})$.
By the description \eqref{stabilizer-entry-condition} of $H_\lambda$ (here
$\lambda_1-\lambda_2=d$, $\lambda_1-\lambda_3=k$, $\lambda_2-\lambda_3=l$, and
the remaining differences are negative), $[Mt^\lambda]=[Nt^\lambda]$ if and only if
\[
1-a_0a_1\in t^d,\qquad B_{13}\in t^k,\qquad\alpha_{1,0}-\alpha_{1,1}\in t^l,
\]
all in $R\llbracket t\rrbracket$, and these hold by the definitions of
$a_1,\alpha_{1,1},\alpha_{2,1}$.  By the uniqueness of Step 3, $(a_1,\alpha_{1,1},\alpha_{2,1})$ are the
$U_1$-coordinates of $x$.  Written with $\pi_l^k$ and $\iota_l^k$ these relations
are exactly \eqref{pos_rel}.  Equivalently, using $a_0a_1\equiv1\pmod{t^d}$, they
read
\begin{equation}\label{eq:positive-overlap}
\begin{aligned}
a_0a_1&\equiv1&&\pmod {t^d},\\
\alpha_{2,1}-a_1\alpha_{1,1}
&\equiv\alpha_{2,0}-a_1\alpha_{1,0}&&\pmod {t^k},\\
\alpha_{1,1}&\equiv\alpha_{1,0}&&\pmod {t^l},\\
\alpha_{2,1}&\equiv\alpha_{2,0}&&\pmod {t^l}.
\end{aligned}
\end{equation}

The matrix $\varphi_{01}$ has the inverse
\[
\varphi_{10}=
\begin{pmatrix}
\iota_l^k&\widetilde z_0(\iota_l^k\pi_l^k-\mathrm{id}_{A_k})\\
0&\pi_l^k
\end{pmatrix}.
\]
Indeed, put $v:=(\iota_l^k\pi_l^k-\mathrm{id})\xi_0=-[\xi_0]_{\ge l}\in t^l A_k$.  Then
$\pi_l^k\eta_1=\eta_0+\pi_l^k(\widetilde z_1v)=\eta_0$, because $\widetilde z_1v$ is
divisible by $t^l$, and, using $\widetilde z_0\widetilde z_1\equiv1\pmod{t^d}$ and
$l+d=k$,
\[
\iota_l^k\xi_1+\widetilde z_0(\iota_l^k\pi_l^k-\mathrm{id})\eta_1
=[\xi_0]_{<l}-\widetilde z_0\widetilde z_1v=[\xi_0]_{<l}-v=\xi_0 .
\]
So $\varphi_{10}\varphi_{01}=\mathrm{id}$, and a left inverse of a square matrix
over the commutative ring $R_0$ is a two-sided inverse.  Hence the two
normal-form charts glue to $\mathcal P_{k,l}$.  Their even and odd coordinate
counts are $d$ and $k+l$, respectively, which proves the final assertions.

The gluing is summarized by the following bundle diagram:
\begin{equation}\label{diag:positive-gluing}
\begin{tikzcd}[
  column sep=5.2em,
  row sep=3.2em,
  every node/.append style={font=\normalfont}]
\begin{gathered}
J_{d-1}(U_{01})\\[-2pt]
\times\Pi(A_k\oplus A_l)
\end{gathered}
  \arrow[rr,"\Phi_{01}"]
  \arrow[d]
&&
\begin{gathered}
J_{d-1}(U_{01})\\[-2pt]
\times\Pi(A_l\oplus A_k)
\end{gathered}
  \arrow[d]
\\
J_{d-1}(U_{01})
  \arrow[rr,equal]
&&
J_{d-1}(U_{01})
\end{tikzcd}
\end{equation}
Here the upper horizontal arrow is given explicitly by \eqref{pos_rel}.

\end{pf}
\begin{theorem}[Mixed-sign-chamber Schubert cell]\label{th21mix}
Let $\lambda=(k,l|0)$ with $k>0>l$, and set $c:=-l>0$ and $d:=k-l=k+c$.
Then $Gr^\lambda$ is represented by a superscheme $\mathcal P'_{k,-l}$ with
two charts
\[
\mathcal U_i'=J_{d-1}(U_i)\times\Pi(A_k\oplus A_{-l}),
\qquad i=0,1,
\]
with coordinates $(\widetilde z_i,\xi_i,\xi_i')$ constructed in Steps~1--3
below.  On the overlap the constant term $\widetilde z_0(0)\in R_0$ is a unit;
put $w:=\widetilde z_0^{-1}\in R_0\llbracket t\rrbracket$ (the inverse power
series).  Write $[f]_{<r}$ and $[f]_{\ge r}$ for the truncations of a power
series $f$ to the terms of degree $<r$ and $\ge r$, respectively.  Then the gluing map
$\Phi_{01}:\mathcal U_0'|_{U_{01}}\to\mathcal U_1'|_{U_{01}}$ is
\begin{equation}\label{mix_rel}
\boxed{
\begin{aligned}
\xi_1&=-[\,w\,\xi_0\,]_{<k},\qquad
\xi_1'=[\,w\,\xi_0'\,]_{<c},\\[2pt]
\widetilde z_1&=\Big[\,w-\xi_1\cdot[\,\xi_0'\,w\,]_{\ge c}\Big]_{<d}.
\end{aligned}}
\end{equation}
In particular the odd part of $\Phi_{01}$ is linear in the odd coordinates
with the diagonal matrix $\varphi_{01}=\operatorname{diag}(-w,\,w)$ (followed by
truncation), whereas the even coordinate satisfies
\begin{equation}\label{mix_even_defect}
\widetilde z_0\widetilde z_1\equiv 1+\xi_0\,[\,\xi_0'\,w\,]_{\ge c}
\pmod{t^d},
\end{equation}
and the right-hand side is in general different from $1$.
Consequently,
\[
(Gr^\lambda)_{\mathrm{red}}\cong J_{d-1}(\mathbb P^1),
\qquad
\dim Gr^\lambda=(d\mid d).
\]
\end{theorem}
\begin{rmk}[Comparison with the first version of this paper]\label{rmk:gens}
The first version of this paper stated the gluing in the simpler form
$\widetilde z_1=\widetilde z_0^{-1}\pmod{t^d}$, $\varphi_{01}=\operatorname{diag}(-\widetilde z_0^{-1},\widetilde z_0^{-1})$.
That formula is correct if all odd coefficients are multiples of a single odd
element (for instance on points with values in $\Lambda(\xi_1)$), because then
the product of two odd series vanishes and the correction term in
\eqref{mix_rel} disappears; it is not correct on points with values in
Grassmann algebras with at least two generators.  Points with values in
$\Lambda(\xi_1)$ detect the odd tangent directions but not the odd--odd
(bilinear) terms of the even transition function, so they cannot decide the
gluing; the functor of points has to be tested on $\Lambda(\xi_1,\xi_2,\dots)$
with at least two generators.  The smallest example is $\lambda=(1,-1|0)$:
here $d=2$, $\xi_0=\xi_{0,0}$ and $\xi_0'=\xi'_{0,0}$ are constants, and with
$\widetilde z_0=z_0+ut$ and $w=w_0+w_1t+\cdots$ one finds
\[
\widetilde z_1=w_0+\bigl(w_1+w_0w_1\,\xi_{0,0}\xi'_{0,0}\bigr)t,
\qquad
\widetilde z_0\widetilde z_1\equiv 1+w_1\,\xi_{0,0}\xi'_{0,0}\,t\pmod{t^2},
\]
which is $\ne1$ as soon as $u\ne0$ and $\xi_{0,0}\xi'_{0,0}\ne0$.
\end{rmk}
\begin{pf}
By part (2) of Lemma \ref{point_norm}, we may start with the simplified
representatives constructed there.  The proof has four steps: a normal form on
$U_0$, its uniqueness, the corresponding statements on $U_1$, and the
computation of the transition map; only the last one differs from the
positive chamber in an essential way.

\smallskip
\noindent\emph{Step 1: normal form on $U_0$.}
On the chart $U_0=\{a_{22}\text{ is invertible}\}$, the reduction proceeds
as follows.  To keep the formulas readable, set
\[
q_0:=\frac{a_{12}}{a_{22}},\qquad
\rho_0:=\frac{\alpha'}{a_{22}},\qquad
\theta_0:=\alpha_1(1+\rho_0\alpha_2)-q_0\alpha_2,
\]
and
\[
\Delta_0:=\det(a)+\alpha'a_{21}
\left(\alpha_2\frac{a_{21}}{a_{22}}-\alpha_1\right).
\]
Then
\begin{align*}
\left[\left( 
 \arraycolsep=2pt\def\arraystretch{1.2}
 \begin{array}{cc|c}
a_{11}t^{k}&a_{12}t^l&\alpha_1  \\ 
a_{21}t^k&a_{22}t^l&\alpha_2  \\ \hline
0&\alpha't^l&1
\end{array} \right)
\left( 
 \arraycolsep=2pt\def\arraystretch{1.5}
 \begin{array}{cc|c}
a_{22}&0&0  \\ 
-a_{21}(1+\dfrac{\alpha_2\alpha'}{a_{22}})t^{k-l}&1&0  \\ \hline
\alpha'a_{21}t^k&0&1
\end{array}
\right)
\right] 
= \left[\left( 
 \arraycolsep=2pt\def\arraystretch{1.5}
 \begin{array}{cc|c}
\Delta_0t^{k}&a_{12}t^l&\alpha_1  \\ 
0& a_{22}t^{l}&\alpha_2  \\ \hline
0&\alpha't^l&1
\end{array} \right)\right]
\end{align*}
Since $\Delta_0$ is invertible, we may divide the first column by it.  We
continue the transformation as follows:
\begin{align*}
= \left[\left( 
 \arraycolsep=2pt\def\arraystretch{1.5}
 \begin{array}{cc|c}
t^{k}&a_{12}t^l&\alpha_1  \\ 
0&a_{22}t^l &\alpha_2  \\ \hline
0&\alpha't^l&1
\end{array} \right) 
 \left( 
 \arraycolsep=2pt\def\arraystretch{1.5}
 \begin{array}{cc|c}
1&0&x \\ 
0&a_{22}^{-1}&-\frac{\alpha_2}{a_{22}}t^{-l} \\ \hline
0&0&1+\dfrac{\alpha'\alpha_{2}}{a_{22}}
\end{array}
\right)
\right]=
\left[\left( 
 \arraycolsep=2pt\def\arraystretch{1.5}
 \begin{array}{cc|c}
t^{k}&q_0t^l&\theta_0+xt^k  \\ 
0&t^l&0  \\ \hline
0&\dfrac{\alpha'}{a_{22}}t^l&1
\end{array} \right)
\right]
\end{align*}
\begin{align*}
=
\left[\left( 
 \arraycolsep=2pt\def\arraystretch{1.5}
 \begin{array}{cc|c}
t^{k}&q_0t^l&\widetilde{\alpha}_0  \\ 
0&t^l&0  \\ \hline
0&\rho_0t^l&1
\end{array} \right)
\left( 
 \arraycolsep=2pt\def\arraystretch{1.5}
 \begin{array}{cc|c}
1&0&0 \\ 
0&1&0  \\ \hline
0&x_0&1
\end{array} \right)
\left( 
 \arraycolsep=2pt\def\arraystretch{1.5}
 \begin{array}{cc|c}
1&y&0 \\ 
0&1&0  \\ \hline
0&0&1
\end{array} \right)
\right]=
\left[\left( 
 \arraycolsep=2pt\def\arraystretch{1.5}
 \begin{array}{cc|c}
t^{k}&a_{0}t^l&\widetilde{\alpha}_0  \\ 
0&t^l&0  \\ \hline
0&\alpha'_0t^l&1
\end{array} \right)\right], 
\end{align*}
Here $\pi_r$ denotes truncation to powers $1,t,\ldots,t^{r-1}$.  We choose
\[
x_0=\rho_0-\pi_{-l}(\rho_0),
\qquad
y=q_0+\widetilde\alpha_0x_0-
\pi_d(q_0+x_0\widetilde\alpha_0).
\]
The resulting coordinates are characterized by
\begin{equation}\label{eq:mixed-U0-coordinates}
\begin{aligned}
\deg_t a_0&<d,
&a_0&\equiv q_0+\widetilde\alpha_0x_0&&\pmod {t^d},\\
\deg_t\widetilde\alpha_0&<k,
&\widetilde\alpha_0&\equiv\theta_0&&\pmod {t^k},\\
\deg_t\alpha_0'&<-l,
&\alpha_0'&\equiv\rho_0&&\pmod {t^{-l}}.
\end{aligned}
\end{equation}
\smallskip
\noindent\emph{Step 2: uniqueness on $U_0$.}
We now write down the matrix form of conjugation by the stabilizer:
\begin{align}\label{mix_U0_cond}
\left( 
 \arraycolsep=2pt\def\arraystretch{1.5}
 \begin{array}{cc|c}
c_{11}&c_{12}t^{k-l}&\gamma_1 t^{k} \\ 
c_{21}t^{l-k}&c_{22}&\gamma_2 t^{l} \\ \hline
\gamma_1't^{-k}&\gamma_2't^{-l}&1
\end{array} \right)
= 
\left( 
 \arraycolsep=2pt\def\arraystretch{1.5}
 \begin{array}{cc|c}
1&-b_{0}+\widetilde{\beta}_0\beta_0'&-\widetilde{\beta}_0 \\ 
0&1 &0\\ \hline
0&-\beta_0'&1
\end{array} \right)
\left( 
 \arraycolsep=2pt\def\arraystretch{1.5}
 \begin{array}{cc|c}
1&a_{0}&\widetilde{\alpha}_0  \\ 
0&1&0  \\ \hline
0&\alpha_0'&1
\end{array} \right).
\end{align}
It follows that the resulting system of representatives is complete.  Indeed,
the stabilizer equations
\begin{equation}\label{mix_U0_coord}
\left\{
\begin{aligned}
\gamma_1t^k&=\widetilde\alpha_0-\widetilde\beta_0,\\
\gamma_2't^{-l}&=\alpha_0'-\beta_0',\\
c_{12}t^d&=a_0-b_0+
\widetilde\beta_0\beta_0'-\widetilde\beta_0\alpha_0'
\end{aligned}
\right.
\quad\Longleftrightarrow\quad
\left\{
\begin{aligned}
\widetilde\alpha_0&\equiv\widetilde\beta_0&&\pmod {t^k},\\
\alpha_0'&\equiv\beta_0'&&\pmod {t^{-l}},\\
a_0&\equiv b_0&&\pmod {t^d}.
\end{aligned}
\right.
\end{equation}
Thus we may write the unique coordinates as
\[
\widetilde z_0:=a_0\in J_{d-1}(U_0),\qquad
\xi_0:=\widetilde\alpha_0\in\Pi A_k,
\qquad
\xi_0':=\alpha_0'\in\Pi A_{-l}.
\]
We introduce the same even coordinates as in Theorem \ref{th21pos}; the odd coordinates are given by the relations above.

\smallskip
\noindent\emph{Step 3: the chart $U_1$.}
Let $P\in G_0(\mathbb C)\subset G\llbracket t\rrbracket$ be the permutation
matrix that exchanges $e_1$ and $e_2$ and fixes the odd basis vector.
Since $P$ lies in $G\llbracket t\rrbracket$, left multiplication by $P$ is an
automorphism of $Gr^\lambda$, and $P^2=1$.  It exchanges the first two rows of
a representative of the shape in part (2) of Lemma \ref{point_norm}: the new
matrix $PA$ has the same shape, with $(PA)_{22}=a_{12}$ and with $\alpha_1$,
$\alpha_2$ interchanged.  Hence $x\in U_1$ if and only if $Px\in U_0$.

Apply Steps 1 and 2 to $Px$: there are unique
$(\widetilde z,\xi,\xi')$ with $[PA\,t^\lambda]=[N(\widetilde z,\xi,\xi')\,t^\lambda]$,
where $N(a_0,\xi,\xi')$ denotes the normal form of Step 1.  Then
$[A\,t^\lambda]=[P\,N(\widetilde z,\xi,\xi')\,t^\lambda]$, and
$PN(\widetilde z,\xi,\xi')$ is the matrix
\begin{equation}\label{mix_M}
M(\widetilde z,\xi,\xi')=
\left(
 \arraycolsep=2pt\def\arraystretch{1.2}
 \begin{array}{cc|c}
0&1&0\\
1&\widetilde z&\xi\\ \hline
0&\xi'&1
\end{array}\right).
\end{equation}
So the matrices \eqref{mix_M}, with $\deg_t\widetilde z<d$, $\deg_t\xi<k$ and
$\deg_t\xi'<-l$, form a complete system of unique representatives on $U_1$.
We take $(\widetilde z_1,\xi_1,\xi_1')\in J_{d-1}(U_1)\times\Pi A_k\times\Pi A_{-l}$
to be the parameters of $M$.  Explicitly, put
\[
q_1:=\frac{a_{22}}{a_{12}},\qquad
\rho_1:=\frac{\alpha'}{a_{12}},\qquad
\theta_1:=\alpha_2(1+\rho_1\alpha_1)-q_1\alpha_1,
\qquad x_1:=\rho_1-\pi_{-l}(\rho_1);
\]
these are obtained from the quantities $q_0,\rho_0,\theta_0,x_0$ of Step 1 by
the substitution $(a_{12},a_{22},\alpha_1,\alpha_2)\mapsto
(a_{22},a_{12},\alpha_2,\alpha_1)$ induced by $P$, and
\begin{equation}\label{eq:mixed-U1-coordinates}
\widetilde z_1\equiv q_1+\xi_1x_1\pmod{t^d},\qquad
\xi_1\equiv\theta_1\pmod{t^k},\qquad
\xi_1'\equiv\rho_1\pmod{t^{-l}}.
\end{equation}

\smallskip
\noindent\emph{Step 4: the transition map on $U_{01}$.}
\emph{(a) The overlap.}
Let $x\in U_0$ have coordinates $(\widetilde z_0,\xi_0,\xi_0')$.  By
\eqref{eq:mixed-U0-coordinates}, $\widetilde z_0\equiv q_0+\xi_0x_0\pmod{t^d}$,
and $x_0$ is divisible by $t^{-l}$ with $-l>0$; hence
$\widetilde z_0(0)=q_0(0)=a_{12}(0)/a_{22}(0)$.  Thus $x\in U_1$ if and only
if $\widetilde z_0(0)$ is a unit of $R_0$.  (The conditions defining $U_0$ and
$U_1$ do not depend on the representative; see Remark \ref{rmk:charts}.)

\smallskip
\emph{(b) The candidate coordinates.}
On $U_{01}$ let $w=\widetilde z_0^{-1}$ and define $\xi_1,\xi_1',\widetilde z_1$
by \eqref{mix_rel}; put $v:=[\xi_0'w]_{\ge c}$ with $c=-l$.  Since $w$ is even,
\begin{equation}\label{mix_y_id}
v+\xi_1'=w\,\xi_0',\qquad
\widetilde z_1\equiv w-\xi_1v\pmod{t^d},\qquad
\xi_1\equiv-w\xi_0\pmod{t^k}.
\end{equation}
Let $N=N(\widetilde z_0,\xi_0,\xi_0')$ and $M=M(\widetilde z_1,\xi_1,\xi_1')$.
A direct computation with the Schur complement (or with the super Cramer rule
\eqref{sup_cram}) gives
\[
M^{-1}=\left(
 \arraycolsep=2pt\def\arraystretch{1.2}
 \begin{array}{cc|c}
-\widetilde z_1+\xi_1\xi_1'&1&-\xi_1\\
1&0&0\\ \hline
-\xi_1'&0&1
\end{array}\right),
\qquad
B:=M^{-1}N=
\left(
 \arraycolsep=2pt\def\arraystretch{1.4}
 \begin{array}{cc|c}
-\widetilde z_1+\xi_1\xi_1'&B_{12}&B_{13}\\
1&\widetilde z_0&\xi_0\\ \hline
-\xi_1'&B_{32}&1-\xi_1'\xi_0
\end{array}\right),
\]
with
\[
\begin{aligned}
B_{12}&=1-\widetilde z_0\widetilde z_1+\widetilde z_0\xi_1\xi_1'-\xi_1\xi_0',\\
B_{13}&=-\widetilde z_1\xi_0+\xi_1\xi_1'\xi_0-\xi_1,\\
B_{32}&=\xi_0'-\widetilde z_0\xi_1'.
\end{aligned}
\]
By the description \eqref{stabilizer-entry-condition} of $H_\lambda$ (here
$\lambda_1-\lambda_2=d$, $\lambda_1-\lambda_3=k$, $\lambda_3-\lambda_2=c$),
$[Mt^\lambda]=[Nt^\lambda]$ if and only if
\[
B_{12}\in t^dR\llbracket t\rrbracket,\qquad
B_{13}\in t^kR\llbracket t\rrbracket,\qquad
B_{32}\in t^cR\llbracket t\rrbracket,
\]
all other entries of $B$ being power series without further conditions.  We
check the three conditions using \eqref{mix_y_id} and the fact that the
square of an odd series vanishes.
\begin{itemize}
\item $B_{32}$: $\xi_1'\equiv w\xi_0'\pmod{t^c}$, so
$B_{32}\equiv\xi_0'-\widetilde z_0w\,\xi_0'=0\pmod{t^c}$.
\item $B_{12}$: modulo $t^d$ we have
$\widetilde z_0\widetilde z_1\equiv 1-\widetilde z_0\xi_1v$, hence
\[
B_{12}\equiv\widetilde z_0\,\xi_1\,(v+\xi_1')-\xi_1\xi_0'
=\widetilde z_0\,\xi_1\,w\,\xi_0'-\xi_1\xi_0'=0\pmod{t^d}.
\]
\item $B_{13}$: modulo $t^k$, using $\widetilde z_1\equiv w-\xi_1v$,
\[
B_{13}\equiv-w\xi_0-\xi_1+\xi_1(v+\xi_1')\xi_0
=(-w\xi_0-\xi_1)+w\,\xi_1\xi_0'\xi_0 .
\]
The first bracket vanishes modulo $t^k$ by definition of $\xi_1$.  In the
second term we may replace $\xi_1$ by $-w\xi_0$ (the error lies in $t^k$), and
$\xi_0\xi_0'\xi_0=-\xi_0\xi_0\xi_0'=0$.  Hence $B_{13}\equiv0\pmod{t^k}$.
\end{itemize}
Note that in each congruence the modulus is the one required by the
corresponding entry of $H_\lambda$; no congruence modulo $t^k$ or $t^c$ is
used to deduce a statement modulo $t^d$ (this is exactly the step that fails
if one tries to prove $\widetilde z_0\widetilde z_1\equiv1$).

Thus $[Mt^\lambda]=[Nt^\lambda]$.  Since $\deg_t\widetilde z_1<d$,
$\deg_t\xi_1<k$ and $\deg_t\xi_1'<c$, the matrix $M$ is a normal form of
Step~3, and by the uniqueness proved there, $(\widetilde z_1,\xi_1,\xi_1')$
are the $U_1$-coordinates of $x$.  This is \eqref{mix_rel}.

\smallskip
\emph{(c) The even defect.}
By \eqref{mix_y_id}, $\widetilde z_0\widetilde z_1\equiv1-\widetilde z_0\xi_1v
\pmod{t^d}$.  Moreover $-\widetilde z_0\xi_1\equiv\xi_0\pmod{t^k}$ and
$v\in t^cR\llbracket t\rrbracket$, so $-\widetilde z_0\xi_1v\equiv\xi_0v
\pmod{t^{k+c}}$; this is \eqref{mix_even_defect}.

\smallskip
\emph{(d) Gluing and reduction.}
The map $\Phi_{01}$ is the composite of the coordinate bijection on $U_{01}$
coming from $U_0$ and the inverse of the one coming from $U_1$; hence it is an
isomorphism of the open subfunctors $\mathcal U_0'|_{U_{01}}$ and
$\mathcal U_1'|_{U_{01}}$ (a cocycle condition is vacuous for two charts), and
the two charts glue to $\mathcal P'_{k,-l}\cong Gr^\lambda$.  Each chart has
$d$ even and $k+(-l)=d$ odd coordinates.  Modulo the ideal generated by the
odd elements, \eqref{mix_rel} reduces to $\widetilde z_1=[\widetilde z_0^{-1}]_{<d}$,
which is the $(d-1)$-jet of the coordinate change $z\mapsto z^{-1}$ of
$\mathbb P^1$.  This proves the theorem (diagrams: see Remark \ref{rmk:mixed-diagrams}).

\smallskip
\emph{(e) The inverse transition.}
Solving $B':=N^{-1}M\in H_\lambda$ in the same way gives the inverse of
\eqref{mix_rel} in the symmetric form
$\xi_0=-[w'\xi_1]_{<k}$, $\xi_0'=[w'\xi_1']_{<c}$,
$\widetilde z_0=[w'-\xi_0[\xi_1'w']_{\ge c}]_{<d}$ with $w'=\widetilde z_1^{-1}$.
We do not need this formula for the proof; it was verified by exact
computer algebra (see the supplementary material).
\end{pf}

\begin{rmk}\label{rmk:mixed-diagrams}
In the positive chamber the gluing is a morphism of vector bundles over the
common even base (diagram \eqref{diag:positive-gluing}).  In the mixed-sign
chamber this is no longer true: the even coordinate $\widetilde z_1$ depends on
the odd coordinates through \eqref{mix_even_defect}, so $\Phi_{01}$ does not
commute with a projection onto the even jet scheme.  What does commute is the
square of zero sections \eqref{diag:mixed-gluing} below.
\end{rmk}

The gluing can be summarized as follows.  Let
$\iota_i:J_{d-1}(U_{01})\hookrightarrow\mathcal U_i'|_{U_{01}}$ be the inclusion of
the reduced chart (all odd coordinates zero) in the coordinate $\widetilde z_i$.  Then
\begin{equation}\label{diag:mixed-gluing}
\begin{tikzcd}[
  column sep=5.2em,
  row sep=3.2em,
  every node/.append style={font=\normalfont}]
J_{d-1}(U_{01})
  \arrow[r,hook,"\iota_0"]
  \arrow[d,equal]
&
\mathcal U_0'|_{U_{01}}
  \arrow[d,"\Phi_{01}"]
\\
J_{d-1}(U_{01})
  \arrow[r,hook,"\iota_1"']
&
\mathcal U_1'|_{U_{01}}
\end{tikzcd}
\end{equation}
commutes: on the reduced charts \eqref{mix_rel} is the jet of $z\mapsto z^{-1}$, which is the identification of the two copies of
$J_{d-1}(U_{01})$ in the coordinates $\widetilde z_0$ and $\widetilde z_1$.

\begin{definition}[Split superscheme]\label{def:split}
Let $X$ be a smooth superscheme, $X_{\mathrm{red}}\subset X$ its reduced
scheme, $\mathcal I$ the ideal sheaf of $X_{\mathrm{red}}$, and
$\mathcal C:=\mathcal I/\mathcal I^2$ the conormal sheaf, a locally free
sheaf on $X_{\mathrm{red}}$ in odd degree; let $E:=\mathcal C^\vee$.  We call $X$ \emph{split} if
there is an isomorphism of superschemes $X\cong\Pi E$ onto the total space of
the odd vector bundle $E$ over $X_{\mathrm{red}}$ that restricts to the
identity on $X_{\mathrm{red}}$, i.e.\ if
$\mathcal O_X\cong\Lambda^\bullet_{\mathcal O_{X_{\mathrm{red}}}}\mathcal C$
as sheaves of superalgebras compatibly with the reduction.
\end{definition}

If $X$ is split and $U\subset X_{\mathrm{red}}$ is an affine space, then $E|_U$ is
trivial, so on $X|_U$ there are coordinates in which the even coordinates are
pulled back from $U$; this is the only consequence of splitness that we use,
namely: \emph{the even coordinates of the split chart differ from any other
coordinates of the chart that reduce to the same functions on $X_{\mathrm{red}}$
by even elements of the ideal generated by the odd coordinates, and in the
split coordinates the transition between two charts acts on the even
coordinates by the reduced transition function alone.}

All the cells of Theorems \ref{th21pos} and \ref{th11} are split (they are
total spaces of odd vector bundles over their reductions, by the explicit
gluing).  In the mixed-sign chamber the situation is different.

\begin{theorem}[Non-splitness of mixed-sign cells]\label{th21nonsplit}
Let $\lambda=(k,l|0)$ with $k>0>l$ and put $c=-l$.
\begin{enumerate}
\item[(a)] If $(k,c)\ne(1,1)$, then $Gr^\lambda$ is not split: the obstruction
already appears at second order in the odd coordinates.
\item[(b)] For $\lambda=(1,-1|0)$ the cell $Gr^\lambda$ is split.
\end{enumerate}
In particular, for all $\lambda$ in the mixed-sign chamber except $(1,-1|0)$,
the supervariety $Gr^\lambda$ is not the total space of an odd vector bundle
over its reduction $J_{d-1}(\mathbb P^1)$, although its reduction and its
odd dimension agree with those of such a bundle.
\end{theorem}
\begin{pf}
\emph{Coordinates.}
Write $\widetilde z_0=\sum_{m<d}u_mt^m$, $\xi_0=\sum_{a<k}\xi_at^a$,
$\xi_0'=\sum_{b<c}\eta_bt^b$, $\widetilde z_1=\sum_{n<d}v_nt^n$ and let
$w=\widetilde z_0^{-1}=\sum_m q_mt^m$; the $q_m$ are polynomials in
$u_0^{-1},u_1,u_2,\dots$ and on $U_{01}$ the functions $q_0^{\pm1},q_1,\dots,q_{d-1}$
are algebraically independent coordinates.  Expanding \eqref{mix_rel} we get,
exactly (no terms of higher odd degree occur),
\begin{equation}\label{eq:nonsplit-transition}
v_n=q_n+\sum_{a<k,\ b<c}g^n_{ab}\,\xi_a\eta_b,\qquad
g^n_{ab}:=\sum_{\substack{i<k,\ j\ge c\\ i+j=n}}q_{i-a}\,q_{j-b},
\qquad(n<d),
\end{equation}
and the odd coordinates $\xi_{1,a}=-\sum_{a'\le a}q_{a-a'}\xi_{a'}$,
$\eta_{1,b}=\sum_{b'\le b}q_{b-b'}\eta_{b'}$ of chart $1$ are linear.
(With $q_m=0$ for $m<0$.)  The coefficient formula
\eqref{eq:nonsplit-transition} follows by expanding \eqref{mix_rel}; it was also
checked against \eqref{mix_rel} by exact computer algebra for several $(k,c)$.

\smallskip
\emph{A necessary condition for splitness.}
Suppose $Gr^\lambda$ is split.  By the consequence of splitness stated above,
there are even elements $f_m$ of the ideal generated by $\xi_a,\eta_b$ in the
chart-$0$ ring, and $\tilde g_n$ in the ideal generated by the odd coordinates
$\xi_{1,a},\eta_{1,b}$ in the chart-$1$ ring, such that
\[
v_n+\tilde g_n=Q_n(u+f)\qquad(n<d)
\]
on the overlap, where $Q_n(u)=q_n$.  Expand to the second order in the odd
coordinates.  The second-order part of $f_m$ is a sum of terms
$\xi\xi$, $\xi\eta$ and $\eta\eta$, and the odd transitions do not mix $\xi$
and $\eta$; hence the $\xi\eta$-components decouple from the others and give
\begin{equation}\label{eq:E-with-functions}
g^n_{a'b'}-\sum_{a\ge a',\,b\ge b'}G^n_{ab}\,q_{a-a'}q_{b-b'}
=-\sum_m (q^2)_{n-m}\,F^m_{a'b'},
\end{equation}
for all $n<d$, $a'<k$, $b'<c$, where $F^m_{ab}\in\mathbb C[u]$ and
$G^n_{ab}\in\mathbb C[v]$ are the coefficients of $\xi_a\eta_b$ and
$\xi_{1,a}\eta_{1,b}$ in $f_m$ and $\tilde g_n$, and we used
$\partial Q_n/\partial u_m=-(q^2)_{n-m}$ and
$\xi_{1,a}\eta_{1,b}=-\sum_{a'\le a,\,b'\le b}q_{a-a'}q_{b-b'}\xi_{a'}\eta_{b'}$.

\smallskip
\emph{Weights.}
The torus $T=\{\operatorname{diag}(s_1,s_2,s_3)\}$ and the loop rotation act on
the cell and on both charts, compatibly with the gluing: the coordinates
$u_m,\ q_m,\ \xi_a,\ \eta_b$ are eigenfunctions of $T\times\mathbb G_m^{\mathrm{rot}}$ with
$T$-weights $(1,-1,0),\ (-1,1,0),\ (1,0,-1),\ (0,-1,1)$ and rotation weights
$m,\ m,\ a,\ b$, respectively (the chart-$1$ coordinates have weights
compatible with the transition functions).  Equation
\eqref{eq:E-with-functions} is homogeneous, so replacing every $F^m_{ab}$ and
$G^n_{ab}$ by its component of the weight dictated by the homogeneity again
gives a solution.  A polynomial in $u$ (respectively in $v$) of $T$-weight $0$
is a constant, because every $u_m$ (respectively $v_n$) has a non-zero $T$-weight
of the same sign.  Therefore we may assume $F^m_{ab}=f_{ab}\delta_{m,a+b}$ and
$G^n_{ab}=h_{ab}\delta_{n,a+b}$ with constants $f_{ab},h_{ab}$, and
\eqref{eq:E-with-functions} becomes
\begin{equation}\label{eq:E}
g^n_{a'b'}+(q^2)_{n-a'-b'}\,f_{a'b'}
=\sum_{\substack{a\ge a',\ b\ge b'\\ a+b=n,\ a<k,\ b<c}}h_{ab}\,q_{a-a'}q_{b-b'},
\end{equation}
an identity of polynomials in the independent variables $q_0,\dots,q_{d-1}$.

\smallskip
\emph{Case $k\ne c$.}
Take $n=d-1$ and $a'=b'=0$.  The right-hand side of \eqref{eq:E} is empty,
since $a+b\le k+c-2=d-2$.  With
$P_r:=\sum_{i<r}q_iq_{d-1-i}$ we have $g^{d-1}_{00}=P_k$ and
$(q^2)_{d-1}=P_k+P_c$, so
\[
(1+f_{00})P_k+f_{00}P_c=0 .
\]
Suppose $k>c$.  Then $P_k=P_c+R$ with $R=\sum_{c\le i<k}q_iq_{d-1-i}$, and
$P_c$, $R$ are non-zero polynomials in disjoint sets of monomials; hence
$(1+2f_{00})P_c+(1+f_{00})R=0$ forces $1+2f_{00}=0$ and $1+f_{00}=0$, a
contradiction.  If $k<c$, then $P_c=P_k+R'$ with
$R'=\sum_{k\le i<c}q_iq_{d-1-i}$, and $(1+2f_{00})P_k+f_{00}R'=0$ forces
$f_{00}=-\tfrac12$ and $f_{00}=0$, again a contradiction.

\smallskip
\emph{Case $k=c\ge2$.}
Take $n=d-1=2k-1$, $a'=0$ and $b'=1$ (possible since $c\ge2$).  Again the
right-hand side of \eqref{eq:E} is empty, and with
$M_i:=q_iq_{2k-2-i}$ (so $M_i=M_{2k-2-i}$, and $M_0,\dots,M_{k-1}$ are distinct monomials)
\[
\sum_{i=0}^{k-1}M_i+f_{01}\sum_{i=0}^{2k-2}M_i=0,
\qquad\text{i.e.}\qquad
(1+2f_{01})\sum_{i=0}^{k-2}M_i+(1+f_{01})M_{k-1}=0 .
\]
Since $k\ge2$, both $\sum_{i\le k-2}M_i$ and $M_{k-1}$ are non-zero and involve
different monomials, so $1+2f_{01}=0=1+f_{01}$, again a contradiction.  This
proves (a).

\smallskip
\emph{Case $(k,c)=(1,1)$.}  Here $d=2$ and \eqref{eq:E} reads
$f_{00}=h_{00}$ (from $n=0$) and $(1+2f_{00})\,q_0q_1=0$ (from $n=1$), so
$f_{00}=h_{00}=-\tfrac12$ is the unique solution.  Conversely, with only two
odd coordinates $\xi:=\xi_0$ and $\eta:=\eta_0$ all products of four odd
elements vanish, and the odd transitions
$\xi_1=-q_0\xi$, $\eta_1=q_0\eta$ are already linear.  The even coordinates
$u_0'=u_0-\tfrac12\xi\eta$, $u_1'=u_1$, $v_0'=v_0-\tfrac12\xi_1\eta_1$,
$v_1'=v_1$ satisfy $v'=Q(u')$ exactly (the second-order computation is exact), so the charts glue
as the total space of an odd vector bundle; this proves (b).
\end{pf}

\begin{rmk}\label{rmk:nonsplit-status}
The proof of (a) uses only two coefficient comparisons in the linear system
\eqref{eq:E}, but the whole system can be solved by exact linear algebra; we
did so for all $(k,c)$ with $k,c\le4$ (see the supplementary material), and it is
inconsistent exactly when $(k,c)\ne(1,1)$.  We also checked, for $k\ne c$ with $k,c\le5$
and for $k=c\le7$, that the single group of equations used in the proof is already
inconsistent.  These computations are consistency checks; the proof above is complete without them.
Theorem~\ref{th21nonsplit} shows that the cells of the mixed-sign chamber
form a family of non-split supervarieties whose reductions are jet schemes of
$\mathbb P^1$.  We do not decide whether such a cell is the jet scheme of
some non-split supervariety.
\end{rmk}

\subsubsection{Geometric examples and odd-mode counts}

The preceding coordinate descriptions make the odd directions particularly
transparent. In the positive chamber the odd part of either local chart is
\begin{equation}\label{eq:positive-odd-jets}
\mathcal E_{k,l}
=\Pi\left(\mathbb C[t]/(t^k)\oplus\mathbb C[t]/(t^l)\right),
\qquad k>l>0,
\end{equation}
whereas in the mixed-sign chamber it is
\begin{equation}\label{eq:mixed-odd-jets}
\mathcal E_{k,l}
=\Pi\left(\mathbb C[t]/(t^k)\oplus\mathbb C[t]/(t^{-l})\right),
\qquad k>0>l.
\end{equation}
Thus the odd dimension is $0|k+l$ in the positive chamber and $0|k-l$ in
the mixed-sign chamber. Since both truncation orders are positive in either
open chamber, reduction modulo $t$ gives
\[
\mathcal E_{k,l}/t\mathcal E_{k,l}\cong
\Pi(\mathbb C\oplus\mathbb C).
\]
In particular, both odd-coordinate families have a constant term in the two
open chambers.

There is a useful boundary calculation. Fix $d>0$ and move along the closed
mixed-sign segment $k-l=d$. If $N_1$ and $N_2$ denote the numbers of
coefficients in the two odd-coordinate families, then
\[
\begin{array}{c|c|c|c|c}
\text{position}&\lambda&N_1&N_2&\dim\mathcal E_\lambda\\ \hline
\text{left boundary}&(0,-d|0)&0&d&0|d\\
\text{interior}&(k,k-d|0),\ 0<k<d&k&d-k&0|d\\
\text{right boundary}&(d,0|0)&d&0&0|d
\end{array}
\]
Hence the total odd dimension stays constant along this segment, although
the coefficients are redistributed between the two odd families. At either
boundary one family disappears.

We record several small families that illustrate the geometry.
\begin{itemize}
\item For $\lambda=(k,k|0)$, the reduced $GL_2$ orbit is a point, while the
odd chart is
\[
\Pi\left(\mathbb C[t]/(t^k)\oplus\mathbb C[t]/(t^k)\right),
\qquad \dim=0|2k.
\]
Thus a zero-dimensional reduced orbit may still carry nontrivial odd
directions; that the cell is $\mathbb A^{0|2k}$ follows from Proposition
\ref{prop:jet-cells}(ii).

\item For $\lambda=(k,k-1|0)$ with $k\geq2$, the reduced orbit is
$\mathbb P^1$. Let $\mathcal I$ be the ideal sheaf of the reduction in the
corresponding supervariety. Thus $\mathcal I/\mathcal I^2$ is the conormal sheaf, generated by the odd coordinates.  The transition matrices in the proof of Theorem
\ref{th21pos} give
\[
\mathcal I/\mathcal I^2\cong
\Pi\left(\mathcal O_{\mathbb P^1}^{\oplus(2k-2)}
\oplus\mathcal O_{\mathbb P^1}(-1)\right).
\]
This example already shows that the odd coefficients need not glue as a
trivial bundle over the reduced orbit.

\item For $\lambda=(k,k-2|0)$ with $k\geq3$, the reduced open orbit is the
first jet scheme
\[
J_1(\mathbb P^1)\cong\operatorname{Tot}\mathcal O_{\mathbb P^1}(2).
\]
Indeed, if $a_0=z+ut$ and $a_1=w+vt$ satisfy $a_0a_1=1\bmod t^2$, then
$w=z^{-1}$ and $v=-z^{-2}u$. The two highest odd coefficients transform by
\[
\begin{pmatrix}\eta_{1,k-2}\\ \eta_{1,k-1}\end{pmatrix}
=-
\begin{pmatrix}w&0\\ v&w\end{pmatrix}
\begin{pmatrix}\xi_{k-2}\\ \xi_{k-1}\end{pmatrix}.
\]
Consequently, higher odd modes can form a nontrivial extension over the even
jet geometry. The mixed-sign example $\lambda=(1,-1|0)$ has the same reduced
open orbit $J_1(\mathbb P^1)$, but its local odd dimension is only $0|2$; by Theorem \ref{th21nonsplit}(b) this particular cell is split, while all other mixed-sign cells are not.
\end{itemize}

\subsubsection{A graph presentation for $GL_{2|1}$}

The lattice description also gives a convenient graphical bookkeeping device.
Let $\mathcal B$ be the graph whose vertices are homothety classes of
$\mathbb C\llbracket t\rrbracket$-lattices in $\mathbb C((t))^2$ and whose
edges are elementary modifications. This is the rank-one lattice tree. Since
homothety forgets determinant valuation, the graph relevant to the affine
Grassmannian of $GL_2$ is its lift
\[
\widetilde{\mathcal B}
=\{(v,\nu):v\in\mathcal B,\ 
\nu\equiv\operatorname{type}(v)\pmod2\},
\]
where $\nu$ is the valuation of the determinant. At the level of reduced
ind-schemes,
\[
(Gr_{GL_{2|1}})_{\mathrm{red}}
\cong Gr_{GL_2}\times Gr_{GL_1},
\qquad Gr_{GL_1}\cong\mathbb Z.
\]
Accordingly, the reduced incidence graph consists of a discrete family of
copies of $\widetilde{\mathcal B}$, one for each $GL_1$ coweight. This is the
graph counterpart of the lattice realization in Proposition
\ref{prop:lattice-model}; the usual lattice description and Schubert closure
order are recalled in \cite{zhu_2016_introduction,peerenboom2021affine}.

\paragraph{What is decorated.}
Let $K=G_0\llbracket t\rrbracket=GL_n\llbracket t\rrbracket\times GL_m\llbracket t\rrbracket$,
acting on $Gr_G$ by left multiplication.  For a dominant $\lambda$ let
$Z_\lambda:=(Gr^\lambda)_{\mathrm{red}}$ be the reduced cell, i.e.\ the closed subscheme
of $Gr^\lambda$ defined by the ideal $\mathcal I_\lambda$ generated by the odd
sections; it is the classical Schubert cell, a single $K$-orbit, whose complex points
are the lattices of relative position $\lambda$ to the standard lattice
(Proposition \ref{Cartan}).  For $GL_{2|1}$ and $\lambda=(k,l|0)$ these are the
vertices $(v,\nu)$ of $\widetilde{\mathcal B}$ with $d=k-l$ and $\nu=k+l$.
The word \emph{decoration} has no established meaning in the classical theory of
the affine Grassmannian; we use it in the following sense, which is a definition
made in this paper.

\begin{definition}[Decoration]\label{def:decoration}
Let $\lambda$ be dominant and $\mathcal C_\lambda:=\mathcal I_\lambda/\mathcal I_\lambda^2$.
By Theorem \ref{gen}, $Gr^\lambda$ is smooth, so $\mathcal C_\lambda$ is a locally free
$\mathcal O_{Z_\lambda}$-module of rank $o_\lambda$ (in odd degree), and $K$ acts on it
because $K$ preserves $Z_\lambda$ and $\mathcal I_\lambda$.  The \emph{decoration} of the
cell $Gr^\lambda$ is the $K$-equivariant vector bundle $\mathcal C_\lambda$ on $Z_\lambda$.
For a point $x\in Z_\lambda(\mathbb C)$, i.e.\ a vertex of the lattice graph, the decoration
at $x$ is the odd tangent space $V_x:=T^{\mathrm{odd}}_xGr^\lambda=\mathcal C_\lambda(x)^\vee$
as a representation of the stabilizer $\operatorname{Stab}_K(x)$.

For $GL_{2|1}$ and $\lambda=(k,l|0)$ let $T=\{\operatorname{diag}(s_1,s_2,s_3)\}$,
$\chi_a=s_a/s_3$, and let $[t^\lambda]$ be the base vertex of the cell.  The
\emph{marks} of the decoration are the pairs $D_a=(N_a,\varepsilon_a)$, $a=1,2$, where
$N_a$ is the dimension of the $T$-weight space of $V_{[t^\lambda]}$ of weight $\chi_a$
(and $\varepsilon_a=+$) or of weight $\chi_a^{-1}$ (and $\varepsilon_a=-$); if
both vanish, no mark is made.  For another vertex $x=\kappa[t^\lambda]$, $\kappa\in K$, the module $V_x$ is identified with
$V_{[t^\lambda]}$ by $\kappa$ and the marks of $x$ are those of $[t^\lambda]$.  This does not depend on $\kappa$: the marks form an
unordered pair, and changing $\kappa$ by an element of the stabilizer can only interchange $\chi_1$ and $\chi_2$, which happens only if
$\lambda_1=\lambda_2$, where $D_1=D_2$.
In words: a decoration marks, at a lattice, the odd directions in which the supercell is
thickened, with their multiplicities and their off-diagonal blocks.
\end{definition}

Since $\mathcal I_\lambda$ contains all odd sections and its even part lies in
$\mathcal I_\lambda^2$, one has canonically
$\mathcal O/\mathcal I_\lambda^2=\mathcal O_{Z_\lambda}\oplus\mathcal C_\lambda$ with
$\mathcal C_\lambda^2=0$: the decoration is exactly the first-order neighbourhood
$Gr^\lambda_{(1)}=\operatorname{Spec}(\mathcal O/\mathcal I^2_\lambda)$ of $Z_\lambda$ in
$Gr^\lambda$, with its $K$-action.  It is an invariant of the cell, not additional input.

\begin{proposition}[Properties of decorations]\label{prop:decoration}
Let $\lambda=(k,l|0)$ be dominant for $GL_{2|1}$, $d=k-l$, $A_rE:=\bigoplus_{0\le j<r}\mathbb C\,t^jE$.
\begin{enumerate}
\item[(a)] \emph{Tangent module and marks.}
$V_{[t^\lambda]}=\bigoplus_{\lambda_a>0}A_{\lambda_a}E_{a3}\oplus\bigoplus_{\lambda_a<0}A_{|\lambda_a|}E_{3a}$,
where $T$ acts on $t^jE_{a3}$ by $\chi_a$ and on $t^jE_{3a}$ by $\chi_a^{-1}$.
Hence $N_a=|\lambda_a|$, $\varepsilon_a=\operatorname{sign}\lambda_a$, and
$\dim Gr^\lambda=(d\mid N_1+N_2)$.  The marks determine $\lambda$ and conversely;
they are the numerical shadow of the module $V_x$.
\item[(b)] \emph{Action of the stabilizer.}
$\operatorname{Stab}_K([t^\lambda])$ consists of the $g=\operatorname{diag}(g_2,g_1)$,
$g_2\in GL_2\llbracket t\rrbracket$, $g_1\in\mathbb C\llbracket t\rrbracket^\times$, with
$(g_2)_{12}\in t^{d}\mathbb C\llbracket t\rrbracket$, and acts on $V_{[t^\lambda]}$ by
\[
\operatorname{Ad}(g)\,t^jE_{a3}=t^j\sum_{a'}(g_2)_{a'a}\,g_1^{-1}\,E_{a'3},\qquad
\operatorname{Ad}(g)\,t^jE_{3a}=t^j\,g_1\sum_{a'}(g_2^{-1})_{aa'}\,E_{3a'},
\]
read modulo the relations defining $V_{[t^\lambda]}$.  The filtration by powers of $t$ is
invariant.  In the mixed-sign chamber ($k>0>l$) the two marks correspond to
the submodules $A_kE_{13}$ and $A_{-l}E_{32}$, and $V=A_kE_{13}\oplus A_{-l}E_{32}$.
In the positive chamber ($k>l>0$), and by symmetry in the negative chamber, neither
$A_kE_{13}$ nor $A_lE_{23}$ is a submodule; the $T$-weight spaces are canonical, but they are not
submodules, and $W=\operatorname{span}\{t^jE_{13}:l\le j<k\}$ is an invariant submodule of rank $d$.
For $k=l>0$ the stabilizer is all of $K$ and $V\cong A_k\otimes\mathbb C^2$ is the
standard module of $GL_2(A_k)$ (twisted by $g_1^{-1}$): the two marks form one multiplet.
\item[(c)] \emph{Symmetries.}
The automorphism $g\mapsto(g^{\mathrm{st}})^{-1}$ maps $Gr^{(k,l|0)}$ to
$Gr^{(-l,-k|0)}$ and carries $V$ to $V'$ with marks
$\bigl((N_2,-\varepsilon_2),(N_1,-\varepsilon_1)\bigr)$; the translation by the central element
$t^c\cdot1$ maps $Gr^{(k,l|p)}$ to $Gr^{(k+c,l+c|p+c)}$ and preserves the modules
(with $N_a=|\lambda_a-\lambda_3|$ in general); the parity-swap isomorphism
$GL_{n|m}\cong GL_{m|n}$ preserves $V$.
\item[(d)] \emph{Classical limit.}
For $m=0$, $\mathcal C_\lambda=0$ and $Gr^\lambda=Z_\lambda$.
\end{enumerate}
\end{proposition}
\begin{pf}
(a) By Proposition \ref{prop:homogeneous-cell}, $T_{[t^\lambda]}Gr^\lambda=\operatorname{Lie}L^+G/\operatorname{Lie}H_\lambda$;
by \eqref{stabilizer-entry-condition} its odd part is spanned by the classes of $t^jE_{ab}$
at the odd positions with $0\le j<d_{ab}$, $d_{ab}>0$ (the coordinates of Theorem
\ref{gen}, Step 4, together with the base directions), which for $\lambda_3=0$ gives the
stated sum.  $T$ acts by $\operatorname{Ad}$.  The superdimension follows from Theorem \ref{gen}.
(b) The description of the stabilizer is \eqref{stabilizer-entry-condition} restricted to $K$, and the formulas
are $\operatorname{Ad}(g)E=gEg^{-1}$ for the block-diagonal $g$.  The $t$-adic filtration is preserved because the
entries of $g$ are power series.  In the mixed-sign chamber $E_{23}$ and $E_{31}$ vanish in $V$ (their
positions are unconstrained in $H_\lambda$), so $\operatorname{Ad}(g)$ maps $A_kE_{13}$ and $A_{-l}E_{32}$ to
themselves.  For $k>l>0$ take $g_2=\left(\begin{smallmatrix}1&0\\1&1\end{smallmatrix}\right)$: then
$\operatorname{Ad}(g)E_{13}=E_{13}+E_{23}\notin A_kE_{13}$; and take
$g_2=\left(\begin{smallmatrix}1&t^d\\0&1\end{smallmatrix}\right)$: then
$\operatorname{Ad}(g)E_{23}=E_{23}+t^dE_{13}$ with $t^dE_{13}\ne0$ in $V$ because $d<k$.  For $j\ge l$ the
$E_{23}$-component of $\operatorname{Ad}(g)t^jE_{13}$ is divisible by $t^l$ and vanishes in $A_lE_{23}$, so $W$
is invariant.  For $k=l$ the stabilizer is $K$ and the formula exhibits $V$ as $A_k^2$ with $GL_2\llbracket t\rrbracket$ acting through $GL_2(A_k)$.
(c) The first statement was shown in the proof of Proposition \ref{prop:jet-cells}(iii); on Lie algebras the
automorphism sends $E_{a3}$ to a multiple of $E_{3a}$.  The central translation commutes with $L^+G$.  The last claim
is clear, since $V$ is the odd tangent space.  (d) is clear.
\end{pf}

By (a) the marks are a way of reading $\lambda$ through Theorem \ref{gen}, so they carry no information beyond
the coweight modulo the centre; the content of a decoration is the module structure in (b) and the identification
with the first-order neighbourhood.  Figure \ref{fig:coweight-chambers} displays the marks by chamber, and Figure
\ref{fig:super-schubert-graph} displays them on the lattice graph.  Remarks \ref{ex:multiplet} and \ref{ex:lifts} give examples,
Remark \ref{rmk:classical} compares the notion with the classical theory, and Remark \ref{rmk:decoration-limits} states what it does not record.

\begin{rmk}[Example: the multiplet at $d=0$]\label{ex:multiplet}
For $\lambda=(k,k|0)$ the reduced cell is a point and $Gr^\lambda\cong\mathbb A^{0|2k}$ (Proposition \ref{prop:jet-cells}(ii)).
The marks are $((k,+),(k,+))$, but the two families are not independent: the stabilizer is $GL_2\llbracket t\rrbracket$ and the
element $g_2=\left(\begin{smallmatrix}1&0\\1&1\end{smallmatrix}\right)$ sends $E_{13}$ to $E_{13}+E_{23}$.  The
decoration is the $GL_2(A_k)$-module $A_k^2$, a free $A_k$-module of rank $2$ on which $GL_2\llbracket t\rrbracket$ acts through
$GL_2(A_k)$; in particular the two weight lines are not submodules, while the filtration by the $t^jA_k^2$ is invariant.  As a vector
space $A_k^2$ is the odd tangent space of $\mathbb A^{0|2k}=J_{k-1}(\mathbb P^{0|2})$ (Proposition \ref{prop:jet-cells}(ii)).
For $GL_{1|1}$, $\lambda=(k|0)$, the decoration at the single vertex is $A_k$ with $(g_1,g_2)$ acting through $g_1g_2^{-1}$; this
is the $k$ odd coordinates of Theorem \ref{th11}.
\end{rmk}

\begin{rmk}[Example: dependence on the determinant lift]\label{ex:lifts}
Marks depend on the lattice $L=(v,\nu)$, not on its homothety class $v$ alone.  Over a vertex at distance $d=2$
the lattices $\nu=2,0,-2$ have relative positions $(2,0),(1,-1),(0,-2)$ and marks $(2,+)$, $\bigl((1,+),(1,-)\bigr)$ and
$(2,-)$, although the classical sphere $S_2$ is the same (Figure \ref{fig:super-schubert-graph}).  A decoration is a function on
$\widetilde{\mathcal B}$ (more invariantly: on the pairs $(v,m)$, $m=\nu-2p$, given by a $(2|1)$-lattice
up to central homothety), constant along each $K$-orbit.
\end{rmk}

\begin{figure}[ht]
\centering
\begin{adjustbox}{max width=\linewidth}
\begin{tikzpicture}[x=3.4cm,y=1.05cm,
  cls/.style={circle,fill=black,inner sep=1.8pt},
  lat/.style={circle,draw=black,fill=white,inner sep=1.6pt},
  edge/.style={draw=black!55,line width=.5pt},
  lab/.style={font=\scriptsize}]
  \def\P{\textcolor{blue!70!black}{+}}
  \def\M{\textcolor{red!70!black}{-}}
  \foreach \d in {0,1,2}{ \node[cls] (v\d) at (\d,-3.6) {}; }
  \draw[black,line width=.9pt] (v0)--(v1)--(v2);
  \node[font=\footnotesize,anchor=east] at (-1.02,-3.6) {$\mathcal B$};
  \node[font=\scriptsize,below=2pt of v0] {$v_0=L_0$\ ($d=0$)};
  \node[font=\scriptsize,below=2pt of v1] {$v_1\in S_1$\ ($d=1$)};
  \node[font=\scriptsize,below=2pt of v2] {$v_2\in S_2$\ ($d=2$)};
  \draw[gray,line width=.5pt] (v1)--++(-30:.22) (v1)--++(30:.22) (v2)--++(-30:.22) (v2)--++(30:.22);
  \node[font=\tiny,gray,anchor=west] at ($(v2)+(.2,.22)$) {other vertices of $S_d$};
  \foreach \d/\nus in {0/{-2,0,2},1/{-1,1},2/{-2,0,2}}{
    \foreach \h in \nus { \node[lat] (p\d\h) at (\d,\h) {}; }
  }
  \foreach \d in {0,1,2}{ \draw[dotted,gray] (v\d)--(\d,-2.55); }
  \draw[edge] (p0-2)--(p1-1) (p00)--(p1-1) (p00)--(p11) (p02)--(p11);
  \draw[edge] (p1-1)--(p2-2) (p1-1)--(p20) (p11)--(p20) (p11)--(p22);
  \foreach \h in {-2,0,2}{\node[font=\scriptsize,anchor=east] at (-1.02,\h) {$\nu=\h$};}
  \node[lab,anchor=east] at (-.07,2) {$(1,1)$: $(1,\P),(1,\P)^\dagger$};
  \node[lab,anchor=east] at (-.07,0) {$(0,0)$: no mark};
  \node[lab,anchor=east] at (-.07,-2) {$(-1,-1)$: $(1,\M),(1,\M)^\dagger$};
  \node[lab,anchor=south] at (1,1.18) {$(1,0)$: $(1,\P)$};
  \node[lab,anchor=north] at (1,-1.18) {$(0,-1)$: $(1,\M)$};
  \node[lab,anchor=west] at (2.07,2) {$(2,0)$: $(2,\P)$};
  \node[lab,anchor=west] at (2.07,0) {$(1,-1)$: $(1,\P),(1,\M)$};
  \node[lab,anchor=west] at (2.07,-2) {$(0,-2)$: $(2,\M)$};
\end{tikzpicture}
\end{adjustbox}
\caption{The lifted lattice graph $\widetilde{\mathcal B}\to\mathcal B$ over a geodesic
$v_0,v_1,v_2$ of the tree of homothety classes (schematic: over $\mathbb C$ every
vertex has infinitely many neighbours). A lattice $L$ is a point $(v,\nu)$, drawn
in the column over $v$ at height $\nu=\operatorname{val}\det$; edges join lattices
related by an inclusion of index one. Each lattice is labelled by its pair
$(k,l)=((\nu+d)/2,(\nu-d)/2)$ and its decoration (Definition~\ref{def:decoration});
$\mathbin{\color{blue!70!black}+}$ and $\mathbin{\color{red!70!black}-}$ denote the
upper-right and lower-left block. The decoration is the same for all lattices with
the same $(d,\nu)$, i.e.\ along each sphere $S_d$, but it changes along a fibre: the
three lattices over $v_2$ carry different marks. $^\dagger$\,At $d=0$ the two
marks form one multiplet of the $GL_2$-action (Remark~\ref{ex:multiplet}); the dependence on $\nu$ is discussed in Remark~\ref{ex:lifts}.}
\label{fig:super-schubert-graph}
\end{figure}

\begin{rmk}[Relation to the classical theory]\label{rmk:classical}
The classical objects are the lattices (the vertices of $\widetilde{\mathcal B}$), the $K$-orbits $Z_\lambda$
(Proposition \ref{Cartan}), and the Schubert spheres $S_d\subset\mathcal B$.  \emph{Preserved:} the stratification by relative
position is that of $Gr_{G_0}$; the decoration is constant along each stratum by $K$-equivariance; for $m=0$ it is trivial.
\emph{Added:} an odd $K$-equivariant bundle on each classical stratum, equal to the first-order neighbourhood data of
the supercell.  \emph{Not refined:} the classical strata are already indexed by $\lambda$, and the marks only repeat $\lambda$.
\emph{Lost:} the closure order between cells (not studied here), the points of $Gr_G$ outside the cells
(Remark \ref{rmk:tails}), and all information beyond first order, in particular the gluing.
Nothing here asserts that a decoration determines the cell.
\end{rmk}

\begin{rmk}[What the decoration does not see]\label{rmk:decoration-limits}
Let $E_\lambda=\mathcal C_\lambda^\vee$ and let $\Pi E_\lambda$ be the associated split supermanifold over $Z_\lambda$.
Then $Gr^\lambda$ and $\Pi E_\lambda$ have isomorphic first-order neighbourhoods, hence the same decoration, and
$Gr^\lambda\cong\Pi E_\lambda$ (over $Z_\lambda$) if and only if $Gr^\lambda$ is split (Definition \ref{def:split}).
This is expressed by the diagram \eqref{diag:first-order},
\begin{equation}\label{diag:first-order}
\begin{tikzcd}[row sep=3.2em,column sep=4.2em,every node/.append style={font=\normalfont}]
Z_\lambda \arrow[r,hook,"\iota"] \arrow[dr,hook,"\iota_E"']
& Gr^\lambda_{(1)} \arrow[r,hook] \arrow[d,"\cong"']
& Gr^\lambda \arrow[d,dashed,"\psi"]
\\
& (\Pi E_\lambda)_{(1)} \arrow[r,hook]
& \Pi E_\lambda
\end{tikzcd}
\end{equation}
whose triangle and left square commute, and for which a dashed arrow $\psi$ that is the identity on $Z_\lambda$
and makes the right square commute exists exactly for split cells.  By Theorem \ref{th21nonsplit}, for every mixed-sign
cell except $\lambda=(1,-1|0)$ the arrow $\psi$ does not exist.  For instance, for $\lambda=(2,-1|0)$ one has $Z_\lambda=J_2(\mathbb P^1)$,
$V=A_2E_{13}\oplus A_1E_{32}$ with marks $\bigl((2,+),(1,-)\bigr)$, and the obstruction is visible in the even
transition function, $\widetilde z_0\widetilde z_1-1\equiv\xi_{0,0}\xi_0'w_1\,t+(\xi_{0,0}\xi_0'w_2+\xi_{0,1}\xi_0'w_1)\,t^2\pmod{t^3}$,
which is non-zero, and by Theorem \ref{th21nonsplit} no change of even coordinates removes all such terms.  The diagram is a restatement of the
definitions and of Theorem \ref{th21nonsplit}, not a new result.
\end{rmk}

\subsection{Super jet schemes}\label{jets}
The reductions of the supervarieties arising in Theorems \ref{th21pos} and
\ref{th21mix} are jet schemes over the projective line; we refer to
\cite{Mourtada2023} for the basic definitions in the classical case.
For supermanifolds, jets and Weil bundles have been considered, for instance, in
\cite{Bruce2014}; so the notion itself is not new.  We use it in the following
functorial form.

For a superscheme $X$ and $r\ge1$, let $J_{r-1}(X)$ be the functor
\[
R\longmapsto X(R\otimes A_r),\qquad A_r=\mathbb C[t]/(t^r),\quad t\ \text{even},
\]
on commutative superalgebras $R$.  For $X=\mathbb A^{p|q}$ one has
$J_{r-1}(X)=\mathbb A^{rp|rq}$, and $J_{r-1}$ takes open immersions to open
immersions; hence $J_{r-1}(X)$ is a superscheme whenever $X$ is covered by
affine superspaces, for instance for the projective superspace
$\mathbb P^{a|b}$.  Here $\mathbb P^{a|b}(S)$ is the set of even
$S$-linear surjections $S^{a+1|b}\to S$ modulo $S_0^\times$ (equivalently, of
tuples $(s_0,\dots,s_a\,|\,\sigma_1,\dots,\sigma_b)$ with $s_i\in S_0$ generating the
unit ideal and $\sigma_j\in S_1$, modulo $S_0^\times$).  For $X=U_i\subset\mathbb P^1$ this is the
jet chart $J_{d-1}(U_i)$ used above.

\begin{theorem}[Cyclic Schubert cells]\label{thcyclic}
Let $G=GL_{n|m}$ and $d\ge1$.  Then
\[
Gr^{(d,0,\dots,0\,|\,0,\dots,0)}\;\cong\;J_{d-1}(\mathbb P^{n-1|m}),
\qquad
Gr^{(0,\dots,0\,|\,d,0,\dots,0)}\;\cong\;J_{d-1}(\mathbb P^{m-1|n}),
\]
of superdimensions $(d(n-1)\mid dm)$ and $(d(m-1)\mid dn)$, respectively.
\end{theorem}
\begin{pf}
Let $\Lambda=R\llbracket t\rrbracket^{n|m}$ be the standard lattice,
$\lambda=(d,0,\dots,0|0,\dots,0)$ and $S=R\otimes A_d=R[t]/(t^d)$.

\emph{The map.}
Let $\varphi\in\mathbb P^{n-1|m}(S)$ be given by an even surjective
$S$-linear form $\widetilde\varphi:S^{n|m}\to S$.  Composing with
$\Lambda\to\Lambda/t^d\Lambda\cong S^{n|m}$ we obtain an
$R\llbracket t\rrbracket$-linear map $\Lambda\to S$ and put
\[
L_\varphi:=\ker(\Lambda\to S)\ \supseteq\ t^d\Lambda .
\]
Replacing $\varphi$ by $u\varphi$ with $u\in S_0^\times$ does not change
$L_\varphi$.  We claim that $L_\varphi$ is a point of $Gr^\lambda(R)$.  This is
a local statement, so we may assume (after a Zariski localization of $R_0$ and
a permutation of the even basis vectors, both of which are harmless for a local statement) that $\widetilde\varphi=(s_1,\dots,s_n|\sigma_1,\dots,\sigma_m)$
with $s_1\in S_0^\times$.  Lift the entries to $\hat s_i\in R_0\llbracket t\rrbracket$,
$\hat\sigma_j\in R_1\llbracket t\rrbracket$; then $\hat s_1$ is a unit.  Let $h$ be the
matrix whose first row is $(\hat s_1,\dots,\hat s_n|\hat\sigma_1,\dots,\hat\sigma_m)$
and whose other rows are those of the identity matrix; it is invertible (its
Berezinian is $\hat s_1$), so $h\in G\llbracket t\rrbracket(R)$.  If $\varphi_0$
is the first-coordinate form, then $\widetilde\varphi=\varphi_0\circ h$ modulo $t^d$, hence
$L_\varphi=h^{-1}\{v:\ v_1\in t^dR\llbracket t\rrbracket\}=h^{-1}t^\lambda\Lambda$,
which lies in the orbit of $[t^\lambda]$.  Thus we obtain a morphism
$\Psi:J_{d-1}(\mathbb P^{n-1|m})\to Gr^\lambda$.

\emph{$\Psi$ is a monomorphism.}
Let $L_\varphi=L_{\varphi'}=:L$.  Both $\widetilde\varphi$ and $\widetilde\varphi'$
induce $R\llbracket t\rrbracket$-linear (hence $S$-linear) isomorphisms
$\Lambda/L\xrightarrow{\ \sim\ }S$; they differ by an even $S$-linear
automorphism of $S$, i.e.\ by a unit of $S_0$.  So $\varphi'=u\varphi$ and
$\varphi=\varphi'$ in $\mathbb P^{n-1|m}(S)$.

\emph{$\Psi$ is locally surjective.}
Locally a point of $Gr^\lambda(R)$ is $L=g\,t^\lambda\Lambda$ with
$g\in G\llbracket t\rrbracket(R)$.  Then $L$ is the kernel of
$v\mapsto(g^{-1}v)_1\bmod t^d$, i.e.\ $L=L_\varphi$ with $\widetilde\varphi$ the first row of
$g^{-1}$ modulo $t^d$; this form is surjective, since $g^{-1}$ is invertible.

Since $J_{d-1}(\mathbb P^{n-1|m})$ is a sheaf in the fpqc topology, $\Psi$ is an
isomorphism of sheaves.  Its source has superdimension $(d(n-1)|dm)$, in
accordance with Theorem \ref{gen}: the even pairs $(1,j)$ with $2\le j\le n$ and the
odd pairs $(1,\hat j)$ with $1\le j\le m$ contribute $d$ each.

For the second isomorphism apply the first one to $GL_{m|n}$ and use the
identification $GL_{n|m}=GL(R^{n|m})=GL(\Pi R^{n|m})=GL_{m|n}$ obtained by
interchanging the parity of the basis; it carries lattices to lattices and the
cell $(0,\dots,0|d,0,\dots,0)$ to $(d,0,\dots,0|0,\dots,0)$.
\end{pf}

\begin{proposition}[Jet-scheme description of the cells of $GL_{1|1}$ and
$GL_{2|1}$]\label{prop:jet-cells}
\begin{enumerate}
\item[(i)] \[Gr^{(k|0)}_{GL_{1|1}}\cong J_{k-1}(\mathbb P^{0|1})=\mathbb A^{0|k},\]
which is Theorem \ref{th11}.
\item[(ii)] For $k>l>0$, with $d=k-l$,
\[Gr^{(k,l|0)}\cong\mathbb A^{0|2l}\times J_{d-1}(\mathbb P^{1|1}).\]
For
$l=0$ this is Theorem \ref{thcyclic} with $n=2$, $m=1$, and for $k=l>0$ one has
$Gr^{(k,k|0)}\cong\mathbb A^{0|2k}=J_{k-1}(\mathbb P^{0|2})$.
\item[(iii)] For $0>k>l$ the cell $Gr^{(k,l|0)}$ is isomorphic to the positive
cell $Gr^{(-l,-k|0)}$, and hence is described by (ii).
\end{enumerate}
\end{proposition}
\begin{pf}
(i) is the first isomorphism of Theorem \ref{thcyclic} for $n=m=1$: $\mathbb P^{0|1}(R\otimes A_k)$ consists of the $k$ odd
coefficients of an odd series, i.e.\ of $\mathbb A^{0|k}(R)$.

(ii) Let $k>l>0$.  In the charts of Theorem \ref{th21pos} put, in chart $0$,
$\theta_0:=([\xi_0]_{<l},\ \eta_0)\in\Pi(A_l\oplus A_l)$ and
$\zeta_0:=t^{-l}[\xi_0]_{\ge l}\in\Pi A_d$, and in chart $1$,
$\theta_1:=(\xi_1,\ [\eta_1]_{<l})\in\Pi(A_l\oplus A_l)$ and
$\zeta_1:=t^{-l}[\eta_1]_{\ge l}\in\Pi A_d$.  By \eqref{pos_rel},
$\xi_1=[\xi_0]_{<l}$ and $\eta_1\equiv\eta_0-\widetilde z_1[\xi_0]_{\ge l}\pmod{t^k}$.
Since $\eta_0$ has degree $<l$ and $[\xi_0]_{\ge l}$ is divisible by $t^l$, this gives
\[
\theta_1=\theta_0,\qquad \zeta_1\equiv-\widetilde z_1\,\zeta_0\pmod{t^d},\qquad
\widetilde z_1=\widetilde z_0^{-1}\pmod{t^d}.
\]
So the $2l$ odd coordinates $\theta$ are the same in both charts (a factor
$\mathbb A^{0|2l}$), and $(\widetilde z_i,\zeta_i)$ transform exactly as the
$(d-1)$-jets of the coordinates $(z_i,\vartheta_i)$ of $\mathbb P^{1|1}$, which are
related by $z_1=z_0^{-1}$, $\vartheta_1=z_1\vartheta_0$, after replacing $\zeta_1$ by $-\zeta_1$.
The case $l=0$ is Theorem~\ref{thcyclic}.  For $k=l>0$ multiply by the central
element $t^{-k}\cdot1$, which is an automorphism of $Gr$ carrying
$(k,k|0)$ to $(0,0|-k)$, and apply the automorphism $g\mapsto(g^{\mathrm{st}})^{-1}$ of
$G$ (see below), which carries $(0,0|-k)$ to $(0,0|k)$; then use
Theorem~\ref{thcyclic} with $n=2$, $m=1$: $Gr^{(0,0|k)}\cong J_{k-1}(\mathbb P^{0|2})=\mathbb A^{0|2k}$.

(iii) The map $g\mapsto(g^{\mathrm{st}})^{-1}$ is an automorphism of $G$ preserving
$G\llbracket t\rrbracket$ and sending $t^\mu$ to $t^{-\mu}$; it maps $Gr^{(k,l|0)}$ to
$Gr^{(-k,-l|0)}=Gr^{(-l,-k|0)}$, the last equality because the permutation matrix
exchanging $e_1,e_2$ lies in $G\llbracket t\rrbracket$.
\end{pf}

\begin{rmk}
The supertransposition satisfies $(g^{\mathrm{st}})^{\mathrm{st}}=\varepsilon g\varepsilon^{-1}$ with
$\varepsilon=\operatorname{diag}(1,\dots,1\,|\,-1,\dots,-1)$; hence
$g\mapsto(g^{\mathrm{st}})^{-1}$ is an automorphism of $G$, but not an involution:
its square is the inner automorphism by the constant matrix
$\varepsilon\in G_0(\mathbb C)$, which preserves every cell.
For the mixed-sign chamber the situation is different:
by Theorem~\ref{th21nonsplit} the cells are non-split, so they are not
of the form in Proposition~\ref{prop:jet-cells}.
\end{rmk}

\section*{Acknowledgements}
 
This work was carried out at the MIPT Center for Pure Mathematics and was financially supported by Russian Science Foundation grant No. 26-11-00115.

\section*{Supplementary material}

The computations that were used to check the gluing formulas of Theorems \ref{th21pos} and \ref{th21mix}, the obstruction system \eqref{eq:E} in the proof of Theorem \ref{th21nonsplit}, the formula for the inverse transition in Step~4(e), and the module structure in Proposition \ref{prop:decoration}(b) are implemented in Python (exact rational arithmetic; Grassmann algebra implemented from scratch; \texttt{sympy} for linear algebra) and are provided as ancillary files together with a README (\texttt{grassmann\_lib.py}, \texttt{check\_gluing.py}, \texttt{check\_obstruction.py}, \texttt{check\_decoration\_module.py}).  They do not replace the proofs in the text.

\section*{Declaration of generative AI and AI-assisted technologies in the manuscript preparation process}

During the preparation of the original version of this work, the authors used ChatGPT (OpenAI) to assist with language editing and proofreading, including the identification of grammatical, spelling, and typographical errors.
During the preparation of the revised version, the authors used Claude (Anthropic) to assist with: the independent re-examination of the statements and proofs of Theorems~\ref{th21mix} and~\ref{th21nonsplit} and of related computations, by exact rational arithmetic in Grassmann algebras and by exact linear algebra; the drafting of the proofs of Theorems~\ref{th21nonsplit} and~\ref{thcyclic} and of Proposition~\ref{prop:jet-cells}, the formulation of Definition~\ref{def:decoration} and Proposition~\ref{prop:decoration} and the redesign of the figures, and the restructuring of the exposition of several other proofs; a literature search used to check the bibliography; and the writing of the supplementary scripts mentioned above.
The authors reviewed and edited the output as needed and take full responsibility for the content of the published article.



\providecommand{\bysame}{\leavevmode\hbox to3em{\hrulefill}\thinspace}
\providecommand{\MR}{\relax\ifhmode\unskip\space\fi MR }
\providecommand{\MRhref}[2]{%
  \href{http://www.ams.org/mathscinet-getitem?mr=#1}{#2}
}
\providecommand{\href}[2]{#2}

\end{document}